\documentclass{article}

\usepackage{amssymb}
\usepackage{amsmath}
\usepackage{latexsym}
\usepackage{color}
\usepackage[margin=1in]{geometry}

\makeatletter

\newtheorem{assumption}{Assumption}
\def\qed{ \ \vrule width.2cm height.2cm depth0cm\smallskip}
\newenvironment{proof}{\noindent {\bf Proof.\/}}{$\qed$\vskip 0.1in}

\newcommand{\we}{\wedge}
\newcommand{\ol}{\overline}
\newcommand{\ul}{\underline}

\newcommand{\ba}{\begin{array}}
\newcommand{\ea}{\end{array}}
\newcommand{\be}{\begin{equation}}
\newcommand{\ee}{\end{equation}}
\newcommand{\bea}{\begin{eqnarray}}
\newcommand{\eea}{\end{eqnarray}}
\newcommand{\beaa}{\begin{eqnarray*}}
\newcommand{\eeaa}{\end{eqnarray*}}

\def\dbE{\mathbb{E}}
\def\dbF{\mathbb{F}}

\def\dbH{\mathbb{H}}

\def\dbN{\mathbb{N}}
\def\dbP{\mathbb{P}}
\def\dbR{\mathbb{R}}
\def\dbS{\mathbb{S}}

\def\dbQ{\mathbb{Q}}

\def\a{\alpha}
\def\b{\beta}
\def\g{\gamma}
\def\d{\delta}
\def\e{\varepsilon}

\def\si{\sigma}
\def\t{\tau}
\def\f{\varphi}
\def\th{\theta}
\def\o{\omega}

\def\G{\Gamma}
\def\D{\Delta}

\def\L{\Lambda}

\def\O{\Omega}
\def\cA{{\cal A}}

\def\cC{{\cal C}}
\def\cD{{\cal D}}
\def\cE{{\cal E}}
\def\cF{{\cal F}}

\def\cH{{\cal H}}
\def\cI{{\cal I}}
\def\cJ{{\cal J}}

\def\cO{{\cal O}}
\def\cP{{\cal P}}
\def\cQ{{\cal Q}}
\def\cR{{\cal R}}
\def\cS{{\cal S}}
\def\cT{{\cal T}}

\def\ch{\textsc{h}}

\def\no{\noindent}

\def\ms{\medskip}

\def\q{\quad}

\def\pa{\partial}
\def\cd{\cdot}
\def\cds{\cdots}

\def\qed{ \hfill \vrule width.25cm height.25cm depth0cm\smallskip}

\newcommand{\basa}{\begin{assumption}}
\newcommand{\easa}{\end{assumption}}

\newcommand{\bas}{\begin{assum}}
\newcommand{\eas}{\end{assum}}

\def\limsup{\mathop{\overline{\rm lim}}}

\def\pa{\partial}

 \def\cd{\cdot}
\def\cds{\cdots}

\def\1{{\bf 1}}

\def\:{\!:\!}
\def\reff#1{{\rm(\ref{#1})}}
\def \proof{{\noindent \bf Proof\quad}}

at 9pt

\begin{document}

\newtheorem{thm}{Theorem}[section]
\newtheorem{lem}[thm]{Lemma}
\newtheorem{cor}[thm]{Corollary}
\newtheorem{prop}[thm]{Proposition}
\newtheorem{rem}[thm]{Remark}
\newtheorem{eg}[thm]{Example}
\newtheorem{defn}[thm]{Definition}
\newtheorem{assum}[thm]{Assumption}

\renewcommand {\theequation}{\arabic{section}.\arabic{equation}}
\def\thesection{\arabic{section}}

\numberwithin{equation}{section}
\numberwithin{thm}{section}

\title{Viscosity Solutions of Fully Nonlinear Elliptic Path Dependent Partial Differential Equations}

\author{Zhenjie {\sc Ren}\footnote{CMAP, Ecole Polytechnique, 91128 Palaiseau Cedex France. Email: ren@cmap.polytechnique.fr. The author is grateful for the support of the grant from RDM-IdF.}}

\maketitle

\begin{abstract}
This paper extends the recent work on path-dependent PDE's to elliptic equations with Dirichlet boundary conditions. We propose a notion of viscosity solution in the same spirit as \cite{ETZ2,ETZ3}, relying on the theory of optimal stopping under nonlinear expectation. We prove a comparison result implying the uniqueness of viscosity solution, and the existence follows from a Perron-type construction using path-frozen PDE's. We also provide an application to a time homogeneous stochastic control problem motivated by an application in finance.
\end{abstract}

\vspace{5mm}

\noindent{\bf Key words:} Viscosity solutions, optimal stopping, path-dependent PDE's, comparison principle, Perron's approach.

\vspace{5mm}

\noindent{\bf AMS 2000 subject classifications:}  35D40, 35K10, 60H10, 60H30.

\vfill\eject

\section{Introduction}

In this paper, we develop a theory of viscosity solutions of elliptic PDE's on the continuous path space, by extending the recent literature on path-dependent PDE's (PPDE) to this context. 

Nonlinear PPDE's appear in various applications, for example, non-Markovian stochastic control problems are naturally related to path-dependent Hamilton-Jacobi-Bellman equations (see \cite{ETZ2}), and non-Markovian stochastic differential games are related to path-dependent Isaacs equations (see \cite{PZ}). PPDE's are also intimately related to the backward stochastic differential equations introduced by Pardoux and
Peng \cite{PP}, and their extension to the second order in \cite{CSTV,STZ}. We refer to the survey paper \cite{RTZ-survey} as an introduction to this new topic. We also refer to the recent applications in \cite{HTT} to establish a representation of the solution of a class of PPDE's in terms of branching diffusions,
and to \cite{MRTZ} for the small noise large deviation results of path-dependent diffusions.

In the existing literature, the authors are all focus on developing the wellposedness theory for parabolic PPDE's. In this paper, we explorer the notion of elliptic PPDE. An elliptic PPDE on the continuous path space $\O$ is of the form:
\be\label{ppde-intro}
G(\cd, u, \pa_\o u, \pa^2_{\o\o}u)(\o)=0,~ \o\in\cQ \subset \O,
\q\mbox{and}\q
u(\o)=\xi(\o),~\o\in\pa \cQ.
\ee
Our notions of the derivatives $\pa_\o$ and $\pa^2_{\o\o}$
are inspired by the calculus developed in Dupire \cite{Dupire} as well as in Cont and Fournie \cite{CF}. Let 
\beaa
\O^e:=\{\o\in\O: \o=\o_{t\we\cd}~\mbox{for some}~t\in\dbR^+\}
\q\mbox{and}\q u:\O^e\rightarrow \dbR,
\eeaa
i.e. $\O^e$ is the subspace of all the paths with flat tails. Denote by $\{u_t\}_{t\in\dbR^+}$ the process $u_t(\o):=u(\o_{t\we\cd})$. According to \cite{Dupire, CF},  one may define the horizontal and vertical derivatives for the process
\be\label{derv Dupire}
\pa_t u_t(\o):=\lim_{h\rightarrow 0}\frac{u_{t+h}(\o_{t\we\cd})-u_t(\o)}{h}\q\mbox{and}\q
\pa_\o u_t(\o):=\lim_{h\rightarrow 0}\frac{u_t(\o)-u_t(\o_\cd+h1_{[t,\infty)})}{h}.
\ee
Also, in \cite{Dupire, CF} the authors proved that a {\it smooth} process satisfies the functional It\^o formula:
\be\label{formal Ito}
d u_t=\pa_t u~ dt + \pa_\o u ~d\o_t +\frac12\pa^2_{\o\o}u ~d\langle \o\rangle_t,~\dbP\mbox{-a.s. for all continuous semimartingale measures}~\dbP.
\ee
Note that in the definition \eqref{derv Dupire} one requires to extend the process $u$ to the set of c\`adl\`ag paths. Although this technical difficulty is addressed and solved in \cite{CF}, it was observed by Ekren, Touzi and Zhang \cite{ETZ1} that it is more convenient to define the derivatives by the It\^o decomposition \eqref{formal Ito}, namely, we call the continuous processes $\L,Z,\G$ the derivatives of the process $u$ if
\beaa
d u_t = \L_t~ dt + Z_t ~d\o_t +\frac12\G_t~d\langle \o\rangle_t,~\dbP\mbox{-a.s. for all continuous semimartingale measures}~\dbP.
\eeaa
In this paper, we follow this idea to define the path derivatives (see Definition \ref{def: derivatives for Omega e} below).
We next restrict our solution space so that all potential solutions $u$ of elliptic PPDE \eqref{ppde-intro} agree with the time-independence property, i.e. $\pa_t u=0$. A function $u:\O^e\rightarrow\dbR$ is called to be time-invariant, if
\beaa
u(\o)=u\big(\o_{\ell(\cd)}\big)\q\mbox{for all}~\o~\mbox{and all}~\mbox{increasing bijection}~\ell:\dbR^+\rightarrow\dbR^+,
\eeaa
i.e. the value of a time-invariant function $u$ is unchanged by any time scaling of path. 
It follows from the definition of the horizontal derivative in \eqref{derv Dupire} that $\pa_t u=0$. Therefore, the time-invariance implies the time-independence, and in this paper we will prove the wellposedness of time-invariant solutions to PPDE \eqref{ppde-intro}.

It is noteworthy that the elliptic PPDE \eqref{ppde-intro} can reduce to be an elliptic PDE (on the real space). Assume that the nonlinearity $G$ in \eqref{ppde-intro} has no dependence on $\o$, $u:\O^e\rightarrow\dbR$ is a smooth solution to \eqref{ppde-intro}, and that there is a function $v:\dbR^d\rightarrow\dbR$ such that $u(\o) = v(\o_\infty)$ for all $\o\in\O^e$. It follows that the path derivatives reduce to the normal derivatives in the real space, i.e. $\pa_\o u(\o) = \pa_x v(\o_\infty),~ \pa^2_{\o\o} u(\o) = \pa^2_{xx} v(\o_\infty)$. Then the function $v$ satisfies the corresponding elliptic PDE:
\be\label{pde-intro}
-G(v,\pa_x v,\pa^2_{xx} v) =0.
\ee
There is an enormously rich literature studying the elliptic PDE \eqref{pde-intro}. In particular, it is known that the solutions to the Dirichlet problem of the equation \eqref{pde-intro} are not always classical (i.e. smooth enough). For example, Nadirashvili and Vladut constructed in \cite{NV} a singular solution to an equation $-G(\pa^2_{xx}v)=0$, where $G$ satisfies the uniform ellipticity condition. A type of weak solutions, viscosity solutions, was introduced by Crandall and Lions \cite{CL} to study the equations like the one \eqref{pde-intro}, and turns out to be very useful. Since the PDE \eqref{pde-intro} is a special case of the PPDE \eqref{ppde-intro}, we are motivated to develop a theory of viscosity solutions to elliptic PPDE's. 

In this paper, we give a definition of viscosity solutions in the context of elliptic PPDE, and then prove the existence and uniqueness of bounded, uniformly continuous and time-invariant viscosity solutions to the PPDE \eqref{ppde-intro} under certain conditions. We try to keep the structure of the paper close to that of Ekren, Touzi and Zhang \cite{ETZ3}, in which the authors studied the viscosity solutions to parabolic PPDE's. As in \cite{ETZ3}, our main idea is to construct a viscosity solution to \eqref{ppde-intro} by an approximation of piecewise smooth solutions provided by the path-frozen PDE's. Further, we prove the viscosity solution we construct is the unique one through a partial comparison result (i.e. the comparison between a viscosity subsolution and a piecewise smooth supersolution). There are new difficulties in the elliptic context, for example, we need to handle the boundary of Dirichlet problem (in particular, the discontinuity of the hitting time of the boundary $\ch_Q$), and we are not allow to apply certain changes of variables (e.g. $\tilde u_t := e^{rt} u_t$), which are quite convenient in the parabolic context. In particular, our argument to verify the uniform continuity of the constructed viscosity solution is new, and quite different from the argument in \cite{ETZ3}.
Since the path-frozen PDE's do not conserve the uniform continuity of the data of the problem, in \cite{ETZ3} the authors requires additional uniform continuity assumptions (see their Assumption 3.5) to ensure the uniform continuity of the constructed viscosity solution. Curiously, we observe  in the elliptic case that the solutions $\th^{\o,\e}$ to the path-frozen PDE's are `almost' (with an error $\e$) uniform continuous in the parameter $\o$, i.e.
\beaa
\big|\theta^{\omega^{1},\e}-\theta^{\omega^{2},\e}\big|
~\leq~ \e+ \rho(2\e)+C_\e\rho \big(d^{e}(\omega^{1},\omega^{2})\big), \q\mbox{for some modulus of continuity $\rho$}
\eeaa
 (see \eqref{eq:theta uni} below for the more accurate result), and this intermediate result leads to the uniform continuity of the constructed viscosity solution without any extra assumptions. By comparing to the parabolic context, we think the above property is {\it intrinsically elliptic}.

We also provide an application of elliptic PPDE to the problem of superhedging a time invariant derivative security under uncertain volatility model. This is a classical time homogeneous stochastic control problem motivated by the application in financial mathematics.

The rest of paper is organized as follows. Section \ref{sec:preliminary} introduces the main notations, as well as the notion of time-invariance, and recalls the result of optimal stopping under non-dominated measures. Section \ref{sec:ppde} defines the viscosity solution of the elliptic PPDE's. Section \ref{sec: main} presents the main results of this paper. In Section \ref{sec: comparison}, we prove the comparison result which implies the uniqueness of viscosity solutions. In Section \ref{sec: Existence} we verify that a function constructed by a Perron-type approach is an viscosity solution, so the existence follows. We present in Section \ref{sec: app} an application of elliptic PPDE in the field of financial mathematics. Finally, we complete some proofs in the appendix, Section \ref{sec: some proofs}.

\section{Preliminary}\label{sec:preliminary}
Let $\Omega:=\left\{ \omega\in C(\mathbb{R}^{+},\mathbb{R}^{d}):\omega_{0}=0\right\} $
be the set of continuous paths starting from the origin, $B$ be the
canonical process, $\mathbb{F} = \{\cF_t\}_{t\in\dbR^+}$ be the filtration generated by $B$,
 $\cT$ be the set of all $\dbF$-stopping times, and $\mathbb{P}_{0}$ be the Wiener measure. 

\no Denote the $L_\infty$-norm on the continuous path space $\O$ by
$\left\Vert \omega\right\Vert _{\infty}~:=~\sup_{s\leq \infty} |\omega_{s}|$.
 Introduce the concatenation of the continuous paths:
\be\label{def:concate}
(\omega\otimes_{t}\omega')(s):=\omega_{s}1_{[0,t)}(s)+(\omega_{t}+\omega'_{s-t})1_{[t,\infty)}(s) \q \mbox{for $\o,\o'\in\O$ and $s,t\in\dbR^+$}.
\ee
Given a random variable $\xi:\Omega\rightarrow\mathbb{R}$ and a process $X:\dbR^+ \times \O \rightarrow \dbR$, we define the shifted random variable and  the shifted process:
\beaa
\xi^{t,\omega}(\omega^{'}) ~:=~ \xi(\omega\otimes_{t}\omega^{'}),\q
X^{t,\omega}(s, \omega^{'}) ~:=~ X(t+s,\omega\otimes_{t}\omega^{'}).
\eeaa
For a $\t\in \cT$, we often write $\xi^{\t,\o}$ (resp. $X^{\t,\o}$) instead of $\xi^{\t(\o),\o}$ (resp. $X^{\t(\o),\o}$) for simplicity.

In this paper, we focus on a subset of $\O$ denoted by $\O^e$, which will be considered as the solution space of elliptic PPDE's. Define
\begin{center}
 $\O^e ~:=~ \left\{ \omega\in\Omega:\omega=\omega_{t\wedge\cdot}~\mbox{for some}~t\geq0\right\}$,\q i.e. the set of all paths with flat tails.
\end{center}
We denote the starting of the flat fail of a path $\o\in \O^e$ by
\beaa
\bar{t}(\omega) ~:=~ \min\left\{ t:\omega=\omega_{t\wedge\cdot}\right\} \q \text{for all}\q \omega\in\Omega^{e}.
\eeaa
Recall the definition of the concatenation in \eqref{def:concate}. For $\o\in\O^e$, $\o'\in\O$ and $\xi: \O\rightarrow\dbR$, we define
\beaa
(\omega\bar{\otimes}\omega')(s) ~:=~ (\omega\otimes_{\bar{t}(\omega)}\omega')(s)
\q\mbox{and}\q
\xi^{\omega}(\omega') ~:=~ \xi^{\bar{t}(\omega),\omega}(\omega') ~=~ \xi(\o\bar\otimes\o').
\eeaa

In our arguments, we will be interested in the subsets in $\O^e$ of some particular form.  Denote by 
\begin{center}
$\mathcal{R}$ the set of all open, bounded and convex subsets of $\mathbb{R}^{d}$ containing $0$.
\end{center}
We are interested in the subsets in $\O^e$ corresponding to $D\in\cR$:
\be\label{Drond}
\cD:=\{\o\in\O^e: \o_t\in D~\mbox{for all}~t\geq 0\}.
\ee
By defining the stopping time
\beaa
\ch_D:=\inf\{t\geq 0:\o_t\notin D\},\q  \mbox{and the set}\q \cH:=\{\ch_D:D\in\cR\},
\eeaa
we may further define the boundary and the cloture of $\cD$:
\beaa
\pa \cD:=\left\{ \omega\in\Omega^{e}:\bar{t}(\omega)=H_{D}(\omega)\right\},
\q 
\mathrm{cl}(\mathcal{D}):=\mathcal{D}\cup\partial\mathcal{D}.
\eeaa

Elliptic equations are devoted to model time-invariant phenomena, and in the path space the time-invariance property can be formulated mathematically as follows.
\begin{defn}
Define the distance on $\O^e$:
\beaa
d^e(\o,\o'):=\inf_{\ell \in\cI}\sup_{t\in\dbR^+}|\o_{\ell (t)}-\o'_t|,\q \mbox{for}\q \o,\o'\in\O^e,
\eeaa
where $\mathcal{I}$ is the set of all increasing bijections from
$\dbR^+$ to $\dbR^+$. We say $\o$ is equivalent to $\o'$, if $d^e(\o,\o')=0$. A function $u$ on $\O^e$ is time-invariant, if $u$ is well defined on the equivalent class, i.e.
$$u(\o)=u(\o')\q \mbox{whenever}\q d^e(\o,\o')=0.$$

For a subset $\cD \subset \O^e$, $C(\cD)$ denotes the set of all functions $\f : \cD \rightarrow \dbR$ continuous with respect to
$d^{e}(\cdot,\cdot)$. The notations $C\left(\cD ; \mathbb{R}^{d}\right)$,
$C\left(\cD ; \mathbb{S}^{d}\right)$ ($\mathbb{S}^{d}$ denotes the set of $d\times d$ symmetric matrices) are also used when we need to emphasize the space in which the functions take values. 

Finally, we say $u\in \mbox{\rm BUC}(\cD)$ if $u:\cD \rightarrow \dbR$ is bounded and uniformly continuous with respect to $d^e(\cdot,\cdot)$, i.e. there exists a modulus of continuity $\rho$ such that
\be\label{BUC}
\big|u(\o^1)-u(\o^2)\big| ~\le ~ \rho(d^e(\o^1,\o^2))\q\mbox{for all}~\o^1,\o^2\in \cD.
\ee
\end{defn}

\begin{rem}\label{concave_rho}
For any modulus of continuity $\rho$, the concave envelop $\hat \rho := {\rm conc}[\rho]$ is still a modulus of continuity for the same function. Thus, without loss of generality, we may assume that moduli of continuity are concave.
\end{rem}

\begin{eg}
Let us show an example of two equivalent paths of which the $L_\infty$-distance is large. Let $(t_i,x_i)\in \dbR^+\times\dbR^d$ for each $1\leq i\leq n$. We denote by 
\be\label{linearintepolate}
\o ~:=~ {\rm Lin}\big\{(0,0),(t_1,x_1),\cdots,(t_n,x_n)\big\}
\ee
the linear interpolation of the points with a flat tail extending to $t=\infty$ ($\o_t = x_n$, for $t\ge t_n$). Then by defining another path
\beaa
\o' ~:=~ {\rm Lin}\big\{(0,0),(t'_1,x_1),\cdots,(t'_n,x_n)\big\},
\eeaa
we clearly have $d^e(\o,\o')=0$ regardless of the choice of $\{t'_i\}_{1\le i\le n}$. However, the $L_\infty$-distance $\|\o-\o'\|_\infty$ can reach $\max_{1\le i,j\le n}|x_i-x_j|$ by choosing a particular sequence $\{t'_i\}_{1\le i\le n}$.
\end{eg}

\begin{eg}\label{eg:timeinv-func}
We show some examples of time-invariant functions:
\begin{itemize}
\item Markovian case: Assume that there exists $\bar u:\dbR^d\rightarrow\dbR$ such that $u(\o)=\bar u(\o_{\bar t(\o)})$. Since $\big|\o^1_{\bar t(\o^1)}-\o^2_{\bar t(\o^2)}\big|\leq d^e(\o^1,\o^2)$ for all $\o^1,\o^2\in\O^e$, $u$ is time-invariant.

\item Maximum dependent case: Assume that there exists $\bar u: \dbR\rightarrow\dbR$ such that $u(\o)=\bar u(\|\o\|_\infty)$. Note that $\|\o\|_\infty=d^e(\o,0)$ and $d^e(\o^1,0)-d^e(\o^2,0)\leq d^e(\o^1,\o^2)$. Thus, $\|\o^1\|_\infty=\|\o^2\|_\infty$ whenever $d^e(\o^1,\o^2)=0$. Consequently, $u$ is time-invariant.

\end{itemize}
\end{eg}

\no Here are some notations useful below:
\begin{itemize}
\item $O_L~:=~ \left\{ x\in\mathbb{R}^{d}:|x|< L\right\}$, and $\ol O_L ~:=~  \left\{ x\in\mathbb{R}^{d}:|x| \le L\right\}$;

\item  $\big[aI_d,bI_d\big]~:=~\big\{ \g \in \mathbb{S}_d : aI_d\leq \g \leq bI_d\big\}$;
\item $\mathbb{H}^{0}\left(E\right)$ denotes the set of all $\mathbb{F}$-progressively
measurable processes taking values in the set $E$, and in particular  $\mathbb{H}^{0}_L:=\mathbb{H}^0\Big(\big[\sqrt{2/L}I_d,\sqrt{2L}I_d\big]\Big)$ for $L>0$;

\item Denote the quadratic variation of the path $\o$ by $\langle \o \rangle_t := |\o_t|^2 - 2\int_0^t \o_s d\o_s$, where $\int_0^\cd \o_s d\o_s$ is the pathwise stochastic integral defined in Karandikar \cite{Kar};

\item Given $\g,\eta \in \dbS^d$, we define $\g:\eta ~:=~ {\rm Trace}[\g\eta] $.

\item Given a function $\f: \O\rightarrow\dbR^d$, we may define the corresponding process 
\be\label{rv2proc}
\f_t(\o) ~:=~ \f(\o_{t\we\cdot}).
\ee
\end{itemize}

\no We next introduce the {\it smooth functions} on the space $\O^e$. First, for every constant $L>0$, we denote by $\cP^L$ the collection of all continuous semimartingale measures $\dbP$ on $\O$ whose drift and diffusion belong to $\dbH^0(\ol O_L)$ and $\dbH^0_L$, respectively. More precisely, let $\tilde\O:=\O\times\O\times\O$ be an enlarged canonical space and $\tilde B:=(B,A,M)$ be the canonical process. A probability measure $\dbP\in\cP^L$ if there exists an extension $\dbQ^{\a,\b}$ of $\dbP$ on $\tilde \O$ such that:
\be\label{defn:PL}
\left.\ba{lll}
&B=A+M,\q A~\mbox{is absolutely continuous,}~M~\mbox{is a martingale,}&\\
&\|\a^\dbP\|_\infty\le L,~\b^\dbP\in \dbH^0_L,\q
\mbox{where}~\a^\dbP_t:=\frac{dA_t}{dt},~\b^\dbP_t:=\sqrt{\frac{d\langle M\rangle_t}{dt}},&
\ea\right.\dbQ^{\a,\b}\mbox{-a.s.}
\ee

\begin{rem}
The definition of $\cP^L$ is slightly different from the one in \cite{ETZ3}, since we urge that the coefficient of diffusion $\b^\dbP \ge \sqrt{\frac{2}{L}}I_d$.
\end{rem}

\no Further, denote $\cP^\infty:=\cup_{L>0}\cP^L$.

\begin{defn}[Smooth time-invariant processes]\label{def: derivatives for Omega e}
Let $D\in\cR$, and recall $\cD\subset \O^e$ defined in \eqref{Drond}.  We say $\f \in C^{2}(\cD)$,
if $\f \in C(\cD)$ and there exist $Z \in C\left(\cD;\mathbb{R}^{d}\right)$,
$\Gamma\in C\left(\cD;\mathbb{S}^{d}\right)$
such that 
\beaa
d \f_t ~=~ Z_t\cdot dB_{t}+\frac{1}{2}\Gamma_t:\left\langle B\right\rangle _{t}\q \mbox{for}~~t\leq \ch_D,\q
\cP^\infty\mbox{-q.s.}
\eeaa
($\f_t$ is defined in \eqref{rv2proc}), where $\cP^\infty$-q.s. means $\mathbb{P}$-a.s. for all $\dbP\in\cP^\infty$. By a direct localization argument, we see that the above $Z$
and $\Gamma$, if they exist, are unique. Denote $\partial_{\omega}u:=Z$ and $\partial_{\omega\omega}^{2}u:=\Gamma$.
\end{defn}

\begin{rem}
In the Markovian case mentioned in Example \ref{eg:timeinv-func}, if the function $\bar u:\dbR^d\rightarrow\dbR$ is in $C^2(D)$, then it follows from the It\^o's formula that $u\in C^2(\cD)$.
\end{rem}
\begin{rem}
In the path-dependent case, Dupire \cite{Dupire} defined derivatives, $\pa_t u$ and $\pa_\o u$, for process $u:\dbR^+\times \O\rightarrow\dbR^d$. In particular, the $t$-derivative is defined as:
$$\pa_t u(s,\o):=\lim_{h\rightarrow 0^+}\frac{u(s+h,\o_{s\we\cdot})-u(s,\o)}{h}.$$
Also, Dupire and other authors, for example \cite{CF}, proved the functional It\^o formula for the processes regular in Dupire's sense:
\beaa
du_s=\pa_t u_sds + \pa_\o u_s \cdot dB_s+\frac{1}{2}\pa^2_{\o\o}u_s:\left\langle B\right\rangle _s,\q \cP^\infty\text{-q.s.}
\eeaa
Note that in the time-invariant case it always holds that $\pa_t u=0$. Consequently, the processes with Dupire's derivatives in $C(\cD)$ are also smooth according to our definition.
\end{rem}

We next introduce the notations of nonlinear expectations.  For  a family of probabilities $\cP$, a measurable set $A\in\cF_\infty$, a random variable $\xi$, we define the capacity $\cC$, the sub-linear expectation $\ol\cE$ and the super-linear expectation $\ul\cE$:
\beaa
\cC^\cP[A]:=\sup_{\dbP\in\cP}\dbP[A], \q \ol\cE^\cP[\xi]:=\sup_{\dbP\in\cP}\dbE^\dbP[\xi], \q \ul\cE^\cP[\xi]:=\inf_{\dbP\in\cP}\dbE^\dbP[\xi].
\eeaa

\no We also define the optimal stopping operator (in other words, the Snell envelop) $\ul\cS$ and $\ol\cS$:
\beaa
\ol \cS^\cP_{t}\left[X\right](\omega):=\sup_{\t \in \cT} \ol\cE^\cP \left[X^{t,\o}_\t \right], \q 
\ul \cS^\cP_t \left[X\right](\omega):=\inf_{\tau\in \cT}\ul\cE^\cP\left[X^{t,\o}_\t\right],
\q\mbox{with the barrier process $X$}.
\eeaa
Recall the family of probabilities $\cP^L$ defined above. For simplicity, we denote
\beaa
\cC^L := \cC^{\cP^L},\q \ol\cE^L := \ol\cE^{\cP^L}, \q \ul\cE^L := \ul\cE^{\cP^L},\q
\ol\cS^L := \ol\cS^{\cP^L},\q \ul\cS^L := \ul\cS^{\cP^L}.
\eeaa
The existing literature gives the following results.

\begin{lem}[Tower property, {\rm Nutz and van Handel \cite{NvH}}]\label{lem: tower property}
For a bounded random variable $\xi$, we have
\beaa
\overline{\mathcal{E}}^{L}\left[\xi\right]
~=~\overline{\mathcal{E}}^{L}\left[\overline{\mathcal{E}}^{L}\big[\xi^{\t(\cd),\cd}\big]\right]\q\text{for all}~~ \tau\in\cT.
\eeaa
\end{lem}

\begin{lem}[Snell envelop characterization, {\rm Ekren, Touzi and Zhang \cite{ETZ4}}]\label{thm:optimal stopping}
Let ~$T\in \dbR^+$, $\ch_{D}\in\mathcal{H}$  and $X\in\mathrm{BUC}(\mathcal{D})$. Denote $\ch := \ch_D \we T$.
Define the Snell envelope and the corresponding first hitting time
of the obstacles:
\[
Y:=\overline{\mathcal{S}}^{L}\left[X_{ \ch\wedge\cdot}\right],\q \tau^{*}:=\inf\left\{ t\geq0:Y_{t}=X_{t}\right\} .
\]
Then $Y\ge X$, $Y_{\tau^{*}}=X_{\tau^{*}}$ and $\tau^{*}$ is an optimal stopping
time, i.e. $Y_0 = \ol\cE^L [X_{\t^*}] $.
\end{lem}

\no It is also important to have the following result, of which the proof can be found in Appendix.

\begin{prop}\label{prop: c1 stopping}
Let $D\in\mathcal{R}$, and denote 
\be\label{Dx}
D^x ~:=~ \{y:x+y\in D\}\q \mbox{for} \q x\in D.
\ee
Assume that $O$ is also in $\mathcal{R}$. 
 Define a sequence of stopping times $\{\ch_{n}\}_{n\in\dbN}$:
\begin{equation}\label{def: Hn}
\ch_{0}=0,\q \ch_{n}:=\inf\left\{ s\geq \ch_{n-1}:B_{s}-B_{\ch_{n-1}}\notin O\right\} ,\q n\geq 1.
\end{equation}
Then we have
\beaa
\left.\ba{lll}
&{\rm (i)}\q \lim_{n\rightarrow\infty}\mathcal{C}^{L}\left[\ch_{n}<T\right]=0~~\mbox{for all}~~T\in\dbR^+,\q 
{\rm (ii)}\q \mathcal{\overline{E}}^{L}\left[\ch_{D}\right]<\infty, &\\
& {\rm (iii)}\q \lim_{T\rightarrow\infty} \sup_{x\in D} \cC^L[\ch_{D^x} > T] =0,\q
{\rm (iv)}\q \lim_{n\rightarrow\infty} \sup_{x\in D}  \cC^{L}\left[\ch_{n}<\ch_{D^x}\right]=0.&
 \ea\right.
\eeaa
\end{prop}

\section{Fully nonlinear elliptic PPDE's}\label{sec:ppde}

\subsection{Definition of viscosity solutions of uniformly elliptic PPDE's}

Let $Q\in\cR$ and consider $\cQ$ ($:=\{\o\in\O^e: \o_t\in Q~\mbox{for all}~t\geq 0\}$) as the domain of Dirichlet problem of the PPDE:
\be\label{PPDE}
\mathcal{L}u(\omega):=-G(\omega,u,\partial_{\omega}u,\partial_{\omega\omega}^{2}u)=0~~\mbox{for}~~ \omega\in\mathcal{Q},\q
u=\xi~~ \mbox{on} ~~\partial\mathcal{Q},
\ee
with nonlinearity $G$ and boundary condition by
$\xi$.

\begin{assum}\label{assump: generator} 
The nonlinearity $G:\Omega\times\mathbb{R}\times\mathbb{R}^{d}\times\mathbb{S}^{d} \rightarrow \dbR$
satisfies:\\
{\rm (i)}\quad $\left|G(\cdot,0,0,0)\right|\leq C_{0}$;\\
{\rm (ii)}\quad $G$ is uniformly elliptic, i.e., there exists $L_{0}>0$ such that for all $(\omega,y,z)$
\[
G(\omega,y,z,\gamma_{1})-G(\omega,y,z,\gamma_{2}) ~\geq~ \frac{1}{L_{0}}I_{d}:(\gamma_{1}-\gamma_{2})\q\text{for all}\q \gamma_{1}\geq\gamma_{2}.
\]
{\rm (iii)}\quad $G$ is uniformly continuous on $\Omega^{e}$ with respect to
$d^{e}(\cdot,\cdot)$, and is uniformly Lipschitz continuous in $(y,z,\gamma)$ with a Lipschitz constant $L_{0}$;\\
{\rm (iv)}\quad $G$ is uniformly decreasing in $y$, i.e. there exists a function
$\lambda:\mathbb{R}\rightarrow\mathbb{R}$ strictly increasing and
continuous, $\lambda(0)=0$, and
\beaa
G(\omega,y_{1},z,\gamma)-G(\omega,y_{2},z,\gamma) ~\geq~ \lambda(y_{2}-y_{1}),~\mbox{for all}~y_2\geq y_1,(\o,z,\g)\in \O^e\times\dbR^d\times\dbS^d.
\eeaa
\end{assum}

For any time-invariant function $u$ on $\O^e$ and $\omega\in\mathcal{Q}$, we define the set of test functions:
\beaa
\left.\ba{lll}
\underline{\mathcal{A}}^{\cP} u(\omega)
:=
\Big\{\f: \varphi\in C^{2}(\mathcal{O}_\e)~\mbox{and}~(\varphi-u^{\omega})_{0}=\underline{\mathcal{S}}^{\cP}_0\left[(\varphi-u^{\omega})_{\ch_{\e}\wedge\cdot}\right]~\mbox{for some}~\e>0\Big\},\\
\mathcal{\overline{A}}^{\cP} u(\omega)
:=
\Big\{\f: \varphi\in C^{2}(\mathcal{O}_\e)~\mbox{and}~(\varphi-u^{\omega})_{0}=\overline{\mathcal{S}}^{\cP}_0\left[(\varphi-u^{\omega})_{\ch_{\e} \wedge\cdot}\right]~\mbox{for some}~\e>0\Big\},
\ea\right.
\mbox{with}~~\ch_\e := \ch_{O_\e}\we\e.
\eeaa
We call $\ch_\e$ a localization of test function $\f$. In particular, we denote $\ol\cA^L:= \ol\cA^{\cP^L}$, $\ul\cA^L:= \ul\cA^{\cP^L}$, as we choose $\cP^L$ as the family of probabilities.
Now, we define the viscosity solutions to the elliptic PPDE (\ref{PPDE}).
\begin{defn}
Let  $\{u_t\}_{t\in\dbR^+}$ be a time-invariant progressively measurable process.\\
{\rm (i)}\quad $u$ is a $\cP$-viscosity subsolution (resp. supersolution) of PPDE (\ref{PPDE}),
if we have for all $\omega\in\mathcal{Q}$ and $\varphi\in\underline{\mathcal{A}}^{\cP}u(\omega)$
(resp. $\varphi\in\mathcal{\overline{A}}^{\cP}u(\omega)$):
\[
-G(\omega,u(\o),\partial_{\omega}\f_0,\partial_{\omega\omega}^{2}\varphi_0)\leq\ (resp.\ \geq)\ 0.
\]
{\rm (ii)}\quad $u$ is a $\cP$-viscosity solution of PPDE (\ref{PPDE}), if $u$ is both a $\cP$-viscosity
subsolution and a $\cP$-viscosity supersolution of PPDE (\ref{PPDE}).
\end{defn}

By very similar arguments as in the proof of Theorem 3.16 and Theorem 5.1 in \cite{ETZ2}, we may easily prove that:

\begin{thm}[Consistency with classical solution]
Let Assumption \ref{assump: generator} hold true and $L>0$. Given a function $u\in C^{2}(\mathcal{Q})$,
then $u$ is a $\cP^L$-viscosity supersolution (resp. subsolution, solution) to PPDE (\ref{PPDE}) if and only
if $u$ is a classical supersolution (resp. subsolution, solution).
\end{thm}

\begin{thm}[Stability]
Let $L>0$, $G$ satisfy Assumption \ref{assump: generator}, and $u\in\mbox{\rm BUC}(\cQ)$. Assume that

\no {\rm (i)}\q for any $\e>0$, there exist $G^\e$ and $u^\e\in\mbox{\rm BUC}(\cQ)$ such that $G^\e$ satisfies Assumption \ref{assump: generator} and $u^\e$ is a $\cP^L$-viscosity subsolution (resp. supersolution) of PPDE \eqref{PPDE} with generator $G^\e$;

\no {\rm (ii)}\q as $\e\rightarrow 0$, $(G^\e,u^\e)$ converge to $(G,u)$ locally uniformly in the following sense: for any $(\o,y,z,\g)\in\O^e\times\dbR\times\dbR^d\times\dbS^d$, there exits $\d>0$ such that
\beaa
\lim_{\e\rightarrow 0}\sup_{(\tilde\o,\tilde y,\tilde z,\tilde\g)\in O_\d(\o,y,z,\g)}\Big[|(G^\e-G)^\o(\tilde\o,\tilde y,\tilde z,\tilde\g)|+|(u^\e-u)^\o(\tilde\o)|\Big]=0,
\eeaa
where we abuse the notation $O_\d$ to denote the $\d$-ball in the corresponding space.

\no Then $u$ is a $\cP^L$-viscosity solution (resp. supersolution) of PPDE \eqref{PPDE} with generator $G$. 
\end{thm}

\subsection{Equivalent definition by semijets}

Following the standard theory of viscosity solutions for PDE's, we may also define viscosity solutions
via semijets. Similar to \cite{RTZ-survey} and \cite{RTZ}, we introduce the notion of semijets in the context of PPDE. First, denote functions:
$$\psi^{\a,\b}(\o)
= \a \cd \o_{\bar t(\o)}
  + \frac12 \b: \o_{\bar t(\o)}\o^\intercal_{\bar t(\o)} .
$$
We next define the sub- and super-jets:
\beaa
\ul\cJ^L u(\o):=\Big\{(\a,\b): ~\psi^{\a,\b} \in \ul\cA^L u(\o)\Big\}\q \mbox{and}\q
\ol\cJ^L u(\o):=\Big\{(\a,\b): ~\psi^{\a,\b} \in \ol\cA^L u(\o)\Big\}.
\eeaa

\begin{prop}\label{prop:equiv-def}
Let $u\in \mbox{\rm BUC}(\cQ)$. Then $u$ is an $\cP^L$-viscosity subsolution (resp. supersolution) of PPDE \eqref{PPDE}, if and only if for any $\o\in\cQ$,
\beaa
-G\big(\o,u(\o),\a,\b\big)\le (\mbox{resp.}~\geq)~0,\q\mbox{for all}~ (\a,\b)\in\ul\cJ^L u(\o)~ (\mbox{resp.}~\ol\cJ^L u(\o)).
\eeaa
\end{prop}
\begin{proof}
The `only if' part is trivial by the definitions. It remains to prove the `if' part. We only show the result for $\cP^L$-viscosity subsolutions, while the result for the supersolution can be proved similarly. Let $\f\in\ul\cA^L u(\o)$ and $\ch_\d(:=\ch_{O_{\d}}\we\d)$ be the corresponding localization. Without loss of generality, we may assume that $\o={\bf 0}$ (i.e. $\o_t=0$ for all $t\in \dbR^+$) and $\f_0=u_0$.  Define:
\beaa
\a:=\pa_\o \f_0~~\mbox{and}~~\b:=\pa^2_{\o\o}\f_0.
\eeaa
Let $\e>0$. Since the processes $\pa_\o\f$ and $\pa^2_{\o\o}\f$ are both continuous, there exists $\d'\le \d$ such that
\beaa
|\pa_\o\f_t -\a|\le \e~~\mbox{and}~~|\pa^2_{\o\o}\f_t - \b|\le \e,
\q\mbox{for}~~ t\le \ch_{O_{\d'}}.
\eeaa
Denote $\b_\e:=\b+(1+2L)\e$. Then, for all $\t\in \cT$ such that $\t\le \ch_{\d'}$, we have
\begin{multline*}
u_0-\ul\cE^L\big[(\psi^{\a,\b_\e}-u)_\t\big]
~=~ \ol\cE^L\big[(u-u_0 - \psi^{\a,\b_\e})_\t\big]
~\le ~ \ol\cE^L\big[(u-\f)_\t\big]+\ol\cE^L\big[(\f-\f_0 - \psi^{\a,\b_\e})_\t\big]\\
~\le~ \ol\cE^L\Big[\int_0^\t (\pa_\o\f_s-\a)dB_s+\frac12\int_0^\t (\pa^2_{\o\o}\f_s-\b_\e)ds\Big]
~\le~ \ol\cE^L\Big[\int_0^\t \big(L|\pa_\o\f_s-\a|+\frac12(\pa^2_{\o\o}\f_s-\b_\e)\big)ds\Big] ~\le~ 0, 
\end{multline*}
where we used the fact that $\f\in\ul\cA^Lu({\bf 0})$ and the definition of $\cP^L$ in \eqref{defn:PL}. Consequently, we obtain $(\a,\b_\e)\in \ul\cJ^Lu({\bf 0})$, and thus
$$-G(0,u(0),\a,\b_\e)\le 0.$$
Finally, thanks to the continuity of $G$, we obtain the desired result by sending $\e\rightarrow 0$.
\end{proof}

\section{Main results}\label{sec: main}

Following Ekren, Touzi and Zhang \cite{ETZ3}, we introduce the path-frozen PDE's:
\be\label{defn:Oepsilon}
(E)_{\epsilon}^{\omega}\ \ \ {\bf L}^{\omega}v:=-G(\o,v,\pa_x v,\pa^{2}_{xx} v)=0
~\mbox{on}~
\ O_\e(\o) := O_\e\cap Q^\o, \q\mbox{with}~~Q^\o:=Q^{\o_{\bar t(\o)}}
\ee
(Recall the notation in \eqref{Dx}).
Note that $\o$ is a parameter rather than a variable in the above PDE. Similar to \cite{ETZ3}, our wellposedness result relies on the following condition on the PDE $(E)_\e^\o$.

\begin{assum}\label{assmp: approx PDE} 
For $\epsilon>0$, $\omega\in\mathcal{Q}$ and $h\in C\big(\partial O_\e(\o)\big)$, we have $\overline{v}=\underline{v}$,
where
\beaa
&\overline{v}(x):=\inf\left\{ w(x):w\in C_0^2(O_\e(\o)),\ {\bf L}^{\omega}w\geq0~\mbox{on}~O_\e(\o),\ w\geq h\text{ on }\partial O_\e(\o)\right\},&\\
&\underline{v}(x):=\sup\left\{ w(x):w\in C_0^2(O_\e(\o)),\ {\bf L}^{\omega}w\leq0~\mbox{on}~O_\e(\o),\ w\leq h\text{ on }\partial O_\e(\o)\right\},&
\eeaa
and $C^2_0(O_\e(\o)):=C^2(O_\e(\o))\cap C\big({\rm cl}(O_\e(\o))\big)$.
\end{assum}

\no In this paper, we call the classical notion of viscosity solution to PDE (see for example \cite{CL}) as Crandall-Lions (C-L) viscosity solution, in order to distinguish the one to PPDE.

\begin{eg}
Assume that $g:\dbS^d\rightarrow\dbR$ is convex, and that the corresponding uniformly elliptic PDE
$${\bf L}w = -g(\pa^2_{xx} w) = 0~\mbox{on}~O,~~w=h~\mbox{on}~\pa O$$
has a C-L viscosity solution. Then according to Caffareli and Cabre \cite{CC} (Theorem 6.6 on page 54), the C-L viscosity solution has the interior $C^2$-regularity.
In particular, this equation satisfies Assumption \ref{assmp: approx PDE}. 
\end{eg}

The rest of the paper is devoted to prove the following two main results.
\begin{thm}[Comparison result]\label{comparison}
Let Assumptions \ref{assump: generator}
and \ref{assmp: approx PDE} hold true, and $u,v\in {\rm BUC}(\cQ)$ be a $\cP^L$-viscosity sub- and super-solution to the PPDE (\ref{PPDE}) for some $L>0$, respectively.  If $u \leq v$ on $\pa\cQ$, then we have $u\leq v$ on $\cQ$.
\end{thm}

\begin{thm}[Wellposedness]\label{wellposedness}
Let Assumptions \ref{assump: generator}
and \ref{assmp: approx PDE} hold true, and  $\xi\in \mbox{\rm BUC}(\pa \cQ)$. Then the PPDE (\ref{PPDE}) has a unique $\cP^L$-viscosity solution in {\rm BUC}($\cQ$) for $L\ge L_0$.
\end{thm}

\section{Comparison result}\label{sec: comparison}

\subsection{Partial comparison}\label{subsec:partial comparison}

Similar to \cite{ETZ3}, we introduce the class of piecewise smooth processes in our time-invariant context.

\begin{defn}\label{Def: overbar C2}{\rm
Let $u:\mathcal{Q}\rightarrow\mathbb{R}$.
We say $u\in\overline{C}^{2}(\mathcal{Q})$, if $u$ is bounded, 
process $\{u_t\}_{t\in\dbR^+}$ is continuous in $t$, and there exists
an increasing sequence of $\mathbb{F}$-stopping times $\left\{ \ch_{n}\right\}_{n\ge 0} $ ($\ch_0=0$) such that

\no {\rm (i)}\quad for each $i \ge 0$ and $\o\in \cQ$, $\D\ch_{i,\o}:=\ch_{i+1}^{\ch_i,\o}-\ch_i(\o)$ is a stopping time in $\cH$
whenever $\ch_{i}(\omega)<\ch_{Q}(\omega)<\infty$, i.e. there is a set $O_{i,\o}\in\cR$ such that $\D \ch_{i,\o} (\o') = \inf\{t:\o'_t\notin O_{i,\o}\}$; 

\no {\rm (ii)}\quad for each $i\ge 0$ and $\o\in\cQ$, we have 
$$u^{\o_{\ch_i\we\cdot}}\in\mbox{\rm BUC}(\cO_{i,\o})\cap C^2(\cO_{i,\o});$$

\no {\rm (iii)}\quad $\left\{ i:\ch_{i}(\omega)<\ch_{Q}(\omega)\right\} $ is finite $\mathcal{P}^\infty$-q.s.
and $\lim_{i\rightarrow\infty}\mathcal{C}_{0}^{L}\left[\ch_{i}^{\omega}<\ch_{Q}^{\omega}\right]=0$
for all $\omega\in\mathcal{Q}$ and $L>0$.

}
\end{defn}

The rest of the subsection is devoted to the proof of the following partial comparison result.

\begin{prop}\label{prop: partial comparison}
Let Assumption \ref{assump: generator}
hold true. Let $u \in\overline{C}^{2}(\mathcal{Q})$, $v \in\mbox{\rm BUC}(\cQ)$ be a $\cP^L$-viscosity
sub- and supersolution of PPDE (\ref{PPDE}) for some $L>0$, respectively.
If $u \leq v $ on $\pa Q$, then $u\leq v$
in ${\rm cl}(\cQ)$. A similar result holds if we exchange the roles of
$u$ and $v$.
\end{prop}

In preparation to the proof of Proposition \ref{prop: partial comparison}, we prove the following lemma.

\begin{lem}\label{lem:omega^*}
Let $T>0$, $D\in\cR$ and $X\in\mbox{\rm BUC}(\cD)$ and non-negative. Denote $\ch:=\ch_D \we T$. Assume that $X_0>\ol\cE^L[X_{\ch}]$, then there exists $\o^*\in\cD$ and $t^*:=\bar t(\o^*)$ such that 
$$X(\o^*)~=~\ol\cS^L_{t^*}\big[X_{\ch\we\cd}\big](\o^*) 
\q\mbox{and}\q 
X(\o^*)~>~0.$$
\end{lem}
\begin{proof}
Denote $Y$ as the Snell envelop of $X_{\ch\we\cd}$, i.e. $Y_t:=\ol\cS_t^L\big[X_{\ch\we\cdot}\big]$. By Lemma \ref{thm:optimal stopping}, the stopping time $\t^*:=\inf\{t:X_t=Y_t\}$ defines an optimal stopping rule. So, we have
$$\ol\cE^L[X_{\t^*}]=Y_0\geq X_0>\ol\cE^L[X_{\ch}].$$
Hence $\{\t^*<\ch\}\neq \phi$. Suppose that $X_{\t^*}=0$ on $\{\t^*<\ch\}$. Then, 
$$
0=X_{\t^*}1_{\{\t^*<\ch\}}(\o)=Y_{\t^*}1_{\{\t^*<\ch\}}(\o) \geq \ol\cE^L\big[(X_\ch)^{\t^*(\o),\o}\big]1_{\{\t^*<\ch\}}(\o)\ge 0.
$$
The last inequality is due to the fact $X\ge 0$. Therefore $X_{\ch}1_{\{\t^*<\ch\}}=0$. It follows that $X_{\t^*}=X_{\ch}$ on $\{\t^*<\ch\}$. Thus, we conclude that
\beaa
X_0\leq Y_0=\ol\cE^L[X_{\t^*}]=\ol\cE^L[X_{\ch}]<X_0.
\eeaa
This contradiction implies that $\{\t^*<\ch,X_{\t^*}>0\}\neq \phi$. Finally, take $\o \in \{\t^*<\ch,X_{\t^*}>0\}$, and then $\o^*:=\o_{\t^*(\o)\we\cdot}$ is a path satisfying the requirements.
\end{proof}

\vspace{3mm}

\no {\bf Proof of Proposition \ref{prop: partial comparison}}\quad
Recall the notation $\ch_{i}$, $\D \ch_{i,\o}$ and $O_{i,\omega}$ in Definition
\ref{Def: overbar C2}. We devide the proof in two steps.

\no \underline{\it Step 1.}\q We first show that
\beaa
(u - v)_{\ch_{i}}^{+}(\omega)
~\leq~\ol{\cE}^{L}\left[\left(u^{\ch_{i},\omega} - v^{\ch_{i},\omega}\right)_{\D \ch_{i,\o}}^{+}\right]
~=~  \ol{\cE}^{L}\left[\left(\left(u_{\ch_{i+1}} - v_{\ch_{i+1}}\right)^+\right)^{\ch_{i},\omega}\right],
\q\mbox{for all}~~i\ge 0, ~\o \in \cQ.
\eeaa
Without loss of generality, we set $i=0$. Assume the contrary, i.e.
\[
(u-v)^{+}(\boldsymbol{0})-\overline{\mathcal{E}}^{L}\left[\left(u-v\right)_{\ch_{1}}^{+}\right]>0.
\]
Denote $X:=(u-v)^{+}$. Since $\lim_{T\rightarrow\infty} \cC^L [\ch_1 \ge T]=0$ (Proposition \ref{prop: c1 stopping}) and $u,v$ are both bounded, there exists $T>0$ such that
\beaa
X_0-\overline{\mathcal{E}}^{L}\left[X_\ch\right]>0,
\q
\mbox{with}~~\ch:=\ch_1\we T.
\eeaa
 Then, by Lemma \ref{lem:omega^*},  there exists $\o^*\in \cO_{0,\bf 0}$ and $t^*:=\bar t(\o^*)$ such that
\be\label{Xpartial}
X(\o^*)=\ol\cS^L_{t^*}[X_{\ch\we\cdot}](\o^*)
\q\mbox{and}\q
X(\o^*)>0.
\ee
Since $u \in\overline{C}^{2}(\mathcal{Q})$, in particular $u \in C^2(\cO_{0,\bf 0})$, we have $\f:=u^{\o^*}\in C^2\big(\cO_{0,\bf 0}^{\o^*}\big)$ (Recall that for a set $D\in\cR$ and $\o\in\O^e$, we define $D^{\o}:=D^{\o_{\bar t(\o)}}$ and correspondingly we have the definition of $\cD^\o$). Together with \eqref{Xpartial}, we get $\f\in \overline{\mathcal{A}}^{L} v(\omega^{*})$. By the $\cP^L$-viscosity supersolution property of $v$ and Assumption \ref{assump: generator}, this implies that
\begin{eqnarray*}
0  \leq  -G\left(\cdot,v,\partial_{\omega}\varphi_0,\partial_{\omega\omega}^{2}\varphi_0 \right)\left( \omega^{*}\right)
  \leq  -G\left(\cdot,u,\partial_{\omega}u,\partial_{\omega\omega}^{2}u\right)\left(\omega^{*}\right)-\lambda\big(X(\omega^*)\big)
  <  -G\left(\cdot,u,\partial_{\omega}u,\partial_{\omega\omega}^{2}u\right)\left(\omega^{*}\right).
\end{eqnarray*}
This is in contradiction with the classical subsolution property of $u$.

\vspace{3mm}
\no \underline{\it Step 2.}\q By the result of Step 1 and the tower property of $\overline{\mathcal{E}}^{L}$ stated in Lemma \ref{lem: tower property}, we have
\beaa
\ol{\cE}^{L}\left[\left(u-v\right)_{\ch_{i}}^{+}\right]
&\le & \ol{\cE}^{L}\left[\left(u-v\right)_{\ch_{i+1}}^{+}\right]
\q\mbox{for all}\q i\ge 0.
\eeaa
It follows by induction that
\beaa
(u - v)^{+}(\boldsymbol{0})~\leq~\overline{\mathcal{E}}^{L}\left[\left(u-v\right)_{\ch_{i}}^{+}\right]
\q \mbox{for all}~~i\geq 1.
\eeaa
Then we obtain
\beaa
(u - v)^{+}(\boldsymbol{0}) 
~\le~ 
\overline{\mathcal{E}}^{L}\left[\left(u-v\right)_{\ch_{Q}}^{+}\right]+\overline{\mathcal{E}}^{L}\left[\left(u-v\right)_{\ch_{i}}^{+}-\left(u-v\right)_{\ch_{Q}}^{+}\right].
\eeaa
By Proposition \ref{prop: c1 stopping}, we have $\lim_{i\rightarrow\infty} \mathcal{C}^{L}\left[\ch_{i}<\ch_{Q}\right] = 0$. Since $u,v$ are both bounded, we have
\beaa
(u -v)^{+}(\boldsymbol{0}) 
~\leq ~\overline{\mathcal{E}}^{L}\left[\left(u -v\right)_{\ch_{Q}}^{+}\right]
~=~0.
\eeaa
\qed

\subsection{The Perron type construction}
Define the following two functions:
\begin{equation}\label{eq:bar u under u}
\overline{u}(\omega):=\inf\left\{ \psi(\omega):\psi\in\overline{\mathcal{D}}_{Q}^{\xi}(\omega)\right\} ,\q \underline{u}(\omega):=\sup\left\{ \psi(\omega):\psi\in\underline{\mathcal{D}}_{Q}^{\xi}(\omega)\right\} ,
\end{equation}
where
\[
\overline{\mathcal{D}}_{Q}^{\xi}(\omega):=\left\{ \psi\in\overline{C}^{2}(\mathcal{Q}^{\omega}):\mathcal{L}^{\omega}\psi\geq0\ \text{on}\ \mathcal{Q},\ \psi\geq\xi^{\omega}\ \text{on}\ \partial\mathcal{Q}\right\} ,
\]
\[
\mathcal{\underline{D}}_{Q}^{\xi}(\omega):=\left\{ \psi\in\overline{C}^{2}(\mathcal{Q}^{\omega}):\mathcal{L}^{\omega}\psi\leq0\ \text{on}\ \mathcal{Q},\ \psi\leq\xi^{\omega}\ \text{on}\ \partial\mathcal{Q}\right\} .
\]
As a direct corollary of Proposition \ref{prop: partial comparison}, we have:
\begin{cor}
Let $L>0$ be constant. Under Assumption \ref{assump: generator}, for all $\cP^L$-viscosity supersolutions (resp. subsolution) $u\in {\rm BUC}(\cQ)$ such that $u\geq \xi$ (resp. $u\leq \xi$) on $\pa\cQ$, we have $u\geq \ul u$ (resp. $u\leq \ol u$) on $\cQ$.
\end{cor}
In order to prove the comparison result of Theorem \ref{comparison}, it remains to show the following result.

\begin{prop}\label{prop:u=u}
Let $\xi\in \mbox{\rm BUC}(\pa \cQ)$.
Under Assumptions \ref{assump: generator} and \ref{assmp: approx PDE}, we have $\ol u=\ul u$.
\end{prop}
The proof of this proposition is reported in Subsection \ref{sect:proof u=u}, and requires the preparations in Subsection \ref{subsection: represent}.

\subsection{Preliminary: HJB equations}\label{subsection: represent}
In this subsection, we recall the relation between HJB equations and stochastic control problems. Recall the constants $L_{0}$ and $C_{0}$ in Assumption \ref{assump: generator}
and consider two functions:
\be\label{eq:over g}
\left.\ba{lll}
\overline{g}(y,z,\gamma)
:=
C_0 + L_{0}\left|z\right|+L_0y^-+\sup_{\beta\in [\sqrt{2/L_{0}}I_{d},\sqrt{2L_{0}}I_{d}]}\frac12\beta^{2}:\gamma,\\
\underline{g}(y,z,\gamma):=-C_0-L_{0}\left|z\right|-L_0y^++\inf_{\beta\in [\sqrt{2/L_{0}}I_{d},\sqrt{2L_{0}}I_{d}]}\frac12\beta^{2}:\gamma.
\ea\right.
\ee
Then for all nonlinearities $G$ satisfying Assumption \ref{assump: generator},
it holds $\underline{g}\leq G\leq\overline{g}$. Consider the HJB equations:
$$
\overline{\bf L}u:=-\overline{g}(u,\pa_x u,\pa^{2}_{xx} u)=0
\q\text{and}\q
\underline{\bf L}u:=-\underline{g}(u,\pa_x u,\pa^{2}_{xx} u)=0.
$$
In the next lemma, we will show that the solutions to the PDE's above with the boundary condition $h_D$ have the stochastic representations: 
\be\label{def:wbar}
\left.\ba{lll}
&\overline{w}(x)
 := 
\sup_{b\in\mathbb{H}^{0}\left(\left[0,L_{0}\right]\right)}\ol{\mathcal{E}}^{L_{0}}\Big[h_{D}(B_{\ch_{D}^{x}})e^{-\int_{0}^{\ch_{D}^{x}}b_{r}dr}+C_{0}\int_0^{\ch^x_D}e^{-\int_{0}^{t}b_{r}dr}dt\Big],&\\
&\underline{w}(x)
 := 
\inf_{b\in\mathbb{H}^{0}\left(\left[0,L_{0}\right]\right)}\underline{\mathcal{E}}^{L_{0}}\Big[h_{D}(B_{\ch_{D}^{x}})e^{-\int_{0}^{\ch_{D}^{x}}b_{r}dr}+C_{0}\int_0^{\ch^x_D}e^{-\int_{0}^{t}b_{r}dr}dt\Big],&
\ea\right.
\ee
where we use the new notation $$\ch^x_D ~:= ~ \ch_{D^x}$$ so as to shorten the formulas.

\begin{lem}\label{lem: s.t. uc} 

Let $h_{D}(x):=\mathcal{\overline{E}}^{L_{0}}\big[v(\ch_D^x,B_{\ch_{D}^x\wedge\cdot})\big]$ for some $v\in\text{\rm BUC}(\mathbb{R}^{+}\times\Omega^{e})$. Then $\ol w$ and $\ul w$ are the unique C-L viscosity solutions in $\mbox{\rm BUC}({\rm cl}(D))$ to the equations $\ol{\bf L} u=0$ and $\ul{\bf L} u=0$, respectively, with the boundary condition $u=h_D$ on $\pa D$.
\end{lem}

\begin{proof}
We claim and will prove in Proposition \ref{prop: estimate hitting} in Appendix that there exists a modulus of continuity $\rho$
such that
\be\label{estimate-hittingtime}
\overline{\mathcal{E}}^{L_{0}}\big[\left|\ch_D^{x_{1}}-\ch_D^{x_{2}}\right|\big]\leq\rho(|x_{1}-x_{2}|).
\ee
Since $v\in\mbox{BUC}(\dbR^+\times\O^e)$, we obtain that
\bea\label{estimateBUC}
|h_D(x_1)-h_D(x_2)|
&\le & 
\ol\cE^{L_0}\Big[|v(\ch_D^{x_1},B_{\ch^{x_1}_D\we\cdot})-v(\ch_D^{x_2},B_{\ch^{x_2}_D\we\cdot})|\Big]\notag \\
&\le & 
\rho\Big(\ol\cE^{L_0}\big[ |\ch_D^{x_1}-\ch_D^{x_2}|\big]  +  \ol\cE^{L_0}\big[ \|B_{{\ch^{x_1}_D}\we\cd}-B_{{\ch^{x_2}_D}\we\cd}\|_\infty\big]\Big),
\eea
where we used the concavity of $\rho$ (recall Remark \ref{concave_rho}) and the Jensen's inequality. Recall the definition of $\cP^L$ (each $\dbP\in\cP^L$  corresponds to a measure $\dbQ^{\a,\b}$ in an extended probability space). We have
\bea\label{estimatemoment}
\dbE^\dbP\big[\|B_{\ch^{x_1}_D\we\cd}-B_{\ch^{x_2}_D\we\cd}\|_{\infty}\big]
&\le &  \dbE^{\dbQ^{\a,\b}} \Big[\big\|\int_0^{\ch^{x_1}_D\we\cd}\a_t dt -\int_0^{\ch^{x_2}_D\we\cd}\a_t dt \big\|_\infty\Big]
+ \dbE^{\dbQ^{\a,\b}} \Big[ \| M_{\ch^{x_1}_D\we\cd}- M_{\ch^{x_2}_D\we\cd} \|_\infty^2\Big]^\frac12 \notag \\
&\le & L_0\ol\cE^{L_0}\big[|\ch^{x_1}_D-\ch^{x_2}_D|\big]
+ \Big(2L_0\ol\cE^{L_0}\big[|\ch^{x_1}_D-\ch^{x_2}_D|\big]\Big)^\frac12,~~\mbox{for all}~\dbP\in\cP^{L_0}.
\eea
In view of \eqref{estimate-hittingtime}, we conclude that $h_D\in {\rm BUC}(\dbR^d)$. Further, since $h_{D}$ is bounded and the control processes $b$ in \eqref{def:wbar} only takes non-negative values, it follows that for $x_{1},x_{2}\in D$,
\begin{eqnarray*}
\left|\overline{w}(x_{1})-\overline{w}(x_{2})\right| 
 &\leq & 
\overline{\mathcal{E}}^{L_{0}}\big[|h_{D}(B_{\ch_{D}^{x_{1}}})-h_{D}(B_{\ch_{D}^{x_{2}}})|\big]+C\overline{\mathcal{E}}^{L_{0}}\big[\left|\ch_{D}^{x_{1}}-\ch_{D}^{x_{2}}\right|\big].
\end{eqnarray*}
Since $h_D\in \mbox{BUC}(\dbR^d)$, by the same arguments in \eqref{estimateBUC} and \eqref{estimatemoment}, we conclude that $\ol w\in \mbox{BUC}({\rm cl}(D))$. Then, by a verification argument, one can easily show that $\ol w$ is the unique C-L viscosity solution to $\overline{\bf L}u=0$ with the boundary condition $h_{D}$ on $\partial D$. Similarly, we may prove the corresponding result for $\ul w$.
\end{proof}

\subsection{Proof of $\ol u=\ul u$}\label{sect:proof u=u}

Recall the two functions $\ol u, \ul u$ defined in (\ref{eq:bar u under u}). In the next lemma, we will use the path-frozen PDE's to construct the functions $\theta_{n}^{\epsilon}$, which will be needed to construct the approximations of $\overline{u}$ and $\underline{u}$ defined in (\ref{eq:bar u under u}). Recall the notation of linear interpolation in \eqref{linearintepolate}. Then 
\begin{itemize}
\item let $\big(x_1, x_2,\cds,x_n \big)\in (\ol O_\e)^n$, ${\rm x}_i:= \sum_{j=1}^i x_j$ and then denote 
\be\label{pi_n}
\pi_n:={\rm Lin}\big\{(0,0),(1,{\rm x}_1),\cdots,(n,{\rm x}_n)\big\}
\ee
 (in particular, note that $\pi_n\in\O^e$);

\item denote $\pi_{n}^{x}:={\rm Lin}\big\{\pi_{n},(n+1, {\rm x}_n+x)\big\}$ for all $x\in \ol O_\e$ (clearly, we have $\pi^x_n\in\O^e$), where we slightly abuse the notation: ${\rm Lin}\big\{\pi_{n},(n+1, {\rm x}_n+x)\big\} = {\rm Lin}\big\{(0,0),(1,{\rm x}_1),\cdots,(n,{\rm x}_n),(n+1, {\rm x}_n+x)\big\} $;

\item define a sequence of stopping times: $\ch_{0}^{x}:=0$,
\bea\label{def:seq_stop}
&\ch_{1}^{x}:=\inf\left\{ t\geq 0:x+B_{t}\notin O_{\epsilon}\right\},\q 
\ch_{i+1}^{x}:=\inf\left\{ t\geq \ch_{i}^{x}:B_{t}-B_{\ch_{i}^x}\notin O_{\epsilon}\right\} \q \mbox{for}~i\geq 1,&\\
&\mbox{and}\q \ch_i^{\o, \pi_n,x} := \ch_i^x \we \ch_{Q^{\o\bar\otimes\pi_{n}^{x}}}.\notag &
\eea
(Recall that $Q^\o$ is defined in \eqref{defn:Oepsilon});

\item given $\o\in\O$, we define
\beaa
\pi^m_n(x,\o):={\rm Lin}\Big\{\pi_n,\big(n+1, {\rm x}_n+x+\o_{\ch_{1}^{x}}\big),\cds, \big(n+m, {\rm x}_n+x+\o_{\ch_{m}^{x}}\big)\Big\} \q\mbox{for all} ~m\ge 1.
\eeaa
\end{itemize}

\no The following lemma plays an essential role in our arguments.

\begin{lem}
\label{lem:frozen construct}
Let Assumption \ref{assump: generator}
hold, and assume that $\left|\xi\right|\leq C_{0}$.
Let $\o\in\cQ$, $|x_{i}|=\epsilon$ for all $i\geq1$, $\pi_{n}$ be defined as in \eqref{pi_n}, and $\o\bar\otimes\pi_{n}^{x}\in\mathcal{Q}$. Then
 
\no {\rm (i)}\q there exist continuous functions $(\pi_{n},x)\mapsto\theta_{n}^{\o,\e}(\pi_{n},x)$, bounded uniformly in $(\epsilon,n)$, such that
\[
\theta_{n}^{\omega,\e}(\pi_{n};\cdot)\text{ is a C-L viscosity solution of }(E)_{\epsilon}^{\omega\bar{\otimes}\pi_{n}},
\]
with boundary conditions:
\[
\begin{cases}
\theta_{n}^{\omega,\e}(\pi_{n};x)=\xi(\omega\bar{\otimes}\pi_{n}^{x}), & |x|<\epsilon \text{ and }x\in\partial Q^{\omega\bar{\otimes}\pi_{n}},\\
\theta_{n}^{\omega,\e}(\pi_{n};x)=\theta_{n+1}^{\omega,\e}(\pi_{n}^{x};0), & |x|=\epsilon \text{ and }x\in Q^{\omega\bar{\otimes}\pi_{n}};
\end{cases}
\]

\no {\rm (ii)}\q moreover, there is a modulus of
continuity $\rho$ and a constant $C_\e>0$ such that for any $\omega^{1},\omega^{2}\in\mathcal{Q}$
\begin{equation}\label{eq:theta uni}
\left|\theta_{0}^{\omega^{1},\e}(0;0)-\theta_{0}^{\omega^{2},\e}(0;0)\right|
~\leq~ \e+ \rho(2\e)+C_\e\rho \big(d^{e}(\omega^{1},\omega^{2})\big)  .
\end{equation}
\end{lem}

\begin{rem}
For the domain $O_\e(\o)$ defined in \eqref{defn:Oepsilon}, a part of its boundary belongs to $\pa Q^{\o}$, while the rest belongs to $\pa O_{\epsilon}$. On $\partial Q^{\o}\cap \pa O_\e(\o)$, we should set the solution to be equal to the boundary condition of
the PPDE. Otherwise, on $\partial O_{\epsilon} \cap \pa O_\e(\o)$, the
value of the solution should be consistent with that of the next piece
of the path-frozen PDE's. The proof of Lemma \ref{lem:frozen construct} is similar to that of Lemma 6.2 in \cite{ETZ3}. However, the stochastic representations and the estimates that we will use are all in the context of the elliptic equations. So it is necessary to present the proof in detail.
\end{rem}

\no In preparation of the proof of Lemma \ref{lem:frozen construct}, we give the following estimate on the C-L viscosity solutions to the path-frozen PDE's. The proof is reported in Appendix.
\ms


\begin{lem}\label{lem: comp on ball}
Fix $D\in\mathcal{R}$. Let $h^{i}:\ \partial D\rightarrow\mathbb{R}$
be continuous ($i=1,2$), $G$ satisfy Assumption \ref{assump: generator},
and $v^{i}$ be the C-L viscosity solutions to the following PDE's:
\[
G(\o^i,v^i,\pa_x v^i,\pa^{2}_{xx} v^i)=0\ \text{on}\ D,\ v^{i}=h^{i}\ \text{on}\ \partial D.
\]
Then we have
\[
(v^1-v^2)(x) ~\leq~ \ol\cE^{L_{0}}\left[\left(h^1-h^2\right)^{+}(x+B_{\ch^x_{D}})\right]+C\rho \Big(d^e(\o^1,\o^2)\Big),
\]
where $\rho$ is a modulus of continuity in $\o$ of the function $G$.
In particular, if $\o^1=\o^2$, then we have 
$$(v^1-v^2)(x)~\leq~ \overline{\mathcal{E}}^{L_{0}}\left[\left(h^1-h^2\right)^{+}(x+B_{\ch^x_{D}})\right].$$
\end{lem} 

\vspace{3mm}
\no {\bf Proof of Lemma \ref{lem:frozen construct}}\quad Since $\e$ is fixed, to simplify the notation, we omit $\e$ in the superscript in the proof. We devide the proof in five steps.

\vspace{2mm}
\no \underline{\it Step 1.}\quad We first prove (i) in the case of $G:=\overline{g}$, where $\overline{g}$ is defined in (\ref{eq:over g}).
For any $N$, denote
$$
\ol\th^\o_{N,N}(\pi_N;0):=\ol{\cE}^{L_0}\left[(\xi_{\ch_Q})^{\o\bar\otimes\pi_N}\right].
$$
We define $\ol{\th}_{N,n}^{\o}(\pi_{n};\cdot)$ as the C-L viscosity solution of the following PDE
\be\label{eq:PDE over g}
-\ol{g}(\th,\pa_x \th,\pa_{xx}^{2}\th)=0  ~\mbox{on}~O_\e(\o\bar\otimes\pi_n),
~~
\th(x)=\ol\th_{N,n+1}^\o(\pi_{n}^{x};0)  ~\mbox{on}~\partial O_\e(\o\bar\otimes\pi_n),~~ \mbox{for all}~n\leq N-1.
\ee
In order to shorten the formulas below, we denote the path
\beaa
&\Pi_N (\o,\pi_n^x, B) ~:=~ \o\bar\otimes\pi_{n}^{N^\o-n}(x,B)\bar\otimes (B_{\ch_{Q^{\o\bar\otimes\pi_{n}^{x}}}\wedge\cdot})^{\ch_{N^\o-n}^x},&\\
&\mbox{with}\q N^\o~:=~ \max\{n\le i\le N: \ch_{i-n}^{x} < \ch_{Q^{\o\bar\otimes\pi_{n}^{x}}} \}.&
\eeaa
By Lemma \ref{lem: s.t. uc} and simple induction, we have the stochastic representation of $\ol{\th}_{N,n}^{\o}(\pi_{n};\cdot)$:
\beaa
\ol{\th}_{N,n}^\o(\pi_{n};x)
=\sup_{b\in\dbH^{0}\left(\left[0,L_0\right]\right)}\ol\cE^{L_0}
\Big[e^{-\int_{0}^{ \ch_{N-n}^{\o,\pi_n,x}}b_{r}dr}
\xi\Big(\Pi_N(\o,\pi_n^x, B)\Big)
+C_{0}\int_{0}^{\ch_{N-n}^{\o,\pi_{n}.x}}e^{-\int_{0}^{s}b_{r}dr}ds\Big],~~\mbox{for}~n\le N-1.
\eeaa
Lemma \ref{lem: s.t. uc} also implies that 
\be\label{theta-continue}
\overline{\theta}_{N,n}^{\epsilon}(\pi_{n};x)~\mbox{is continuous
in both variables}~(\pi_{n},x),
\ee
and clearly, they are uniformly bounded. We next define
\beaa
\ol{\th}_n^\o(\pi_{n};x)
:=
\sup_{b\in\mathbb{H}^{0}\left(\left[0,L_{0}\right]\right)}\ol{\cE}^{L_0}\Big[e^{-\int_0^{\ch_{Q^{\o\bar\otimes\pi_{n}^{x}}}}b_{r}dr}\limsup_{N\rightarrow\infty}
\xi\Big(\Pi_N(\o,\pi_n^x, B)\Big)
+C_{0}\int_{0}^{\ch_{Q^{\o\bar\otimes\pi_{n}^{x}}}}e^{-\int_{0}^{s}b_{r}dr}ds\Big].
\eeaa
Then it follows that
\begin{eqnarray*}
|\ol{\th}_n^\o(\pi_{n};x)-\ol{\th}_{N,n}^{\o}(\pi_{n};x)| 
& \leq & 
C\mathcal{C}^{L_{0}}\left[\ch_{N-n}^{x}<\ch_{Q^{\o\bar\otimes\pi_{n}^{x}}}\right]\rightarrow0,\ N\rightarrow\infty.
\end{eqnarray*}
By Proposition \ref{prop: c1 stopping}, the convergence
is uniform in $(\pi_{n},x)$. Together with \eqref{theta-continue}, it implies that $\overline{\theta}_{n}^{\o}(\pi_{n};x)$
is uniformly bounded and continuous in $(\pi_{n},x)$. Moreover, by
the stability of C-L viscosity solutions we see that $\overline{\theta}_{n}^{\o}(\pi_{n};\cdot)$
is the C-L viscosity solution of PDE (\ref{eq:PDE over g}) in $O_\e(\o\bar\otimes\pi_n)$,
with the boundary condition:
\beaa
\begin{cases}
\bar{\theta}_{n}^{\o}(\pi_{n};x)=\xi(\o\bar\otimes\pi_{n}^{x}), & |x|<\epsilon\text{ and }x\in\partial Q^{\o\bar\otimes\pi_{n}},\\
\bar{\theta}_{n}^{\o}(\pi_{n};x)=\bar{\theta}_{n+1}^{\o}(\pi_{n}^{x};0), & |x|=\epsilon\text{ and }x\in Q^{\o\bar\otimes\pi_{n}}.
\end{cases}
\eeaa
Hence, we have showed the desired result in the case $G=\overline{g}$. Similarly, we may show that $\underline{\theta}_{n}^{\o}$ defined below is the C-L viscosity solution to the path-frozen PDE when
the nonlinearity is $\underline{g}$:
\beaa
\ul{\theta}^\o_n(\pi_{n};x)
:=
\inf_{b\in\mathbb{H}^{0}\left(\left[0,L_{0}\right]\right)}\underline{\mathcal{E}}^{L_{0}}\Big[e^{-\int_{0}^{\ch_{Q^{\o\bar\otimes\pi_{n}^{x}}}}b_{r}dr}\limsup_{N\rightarrow\infty}
\xi\Big(\Pi_N(\o,\pi_n^x, B)\Big)
+C_{0}\int_{0}^{\ch_{Q^{\o\bar\otimes\pi_{n}^{x}}}}e^{-\int_{0}^{s}b_{r}dr}ds\Big].
\eeaa

\vspace{3mm}

\no \underline{\it Step 2.}\q We next prove (ii) in the case of $G=\ol g$. Considering $\pi_{n}^{x}\in\mathcal{Q}^{\omega^{1}}\cap\mathcal{Q}^{\omega^{2}}$, we have the following estimate:
\begin{multline*}
\Big|\overline{\theta}_{N,n}^{\omega^{1}}(\pi_{n};x)-\overline{\theta}_{N,n}^{\omega^{2}}(\pi_{n};x)\Big| 
\leq  C\overline{\mathcal{E}}^{L_{0}}\left[\left|\ch_{N-n}^{\omega^{1},\pi_{n},x}-\ch_{N-n}^{\omega^{2},\pi_{n},x}\right|\right]\\
+C\overline{\mathcal{E}}^{L_{0}}\biggl[ \Big|
\xi\Big( \Pi_N(\o^1,\pi_n^x, B) \Big) -\xi\Big( \Pi_N(\o^2,\pi_n^x, B) \Big) \Big| \biggr].
\end{multline*}
We observe that
\beaa
&\left|\ch_{N-n}^{\omega^{1},\pi_{n},x}-\ch_{N-n}^{\omega^{2},\pi_{n},x}\right| 
~\leq ~ \left|\ch_{Q^{\omega^{1}\bar\otimes\pi_{n}^{x}}}-\ch_{Q^{\omega^{2}\bar\otimes\pi_{n}^{x}}}\right|,&\\
& d^e\Big( \Pi_N(\o^1,\pi_n^x, B),\Pi_N(\o^2,\pi_n^x, B)\Big) 
~ \le ~ d^e(\o^1,\o^2) + \Big\|B_{\ch_{Q^{\o^1\bar\otimes\pi_{n}^{x}}}\we\cd} - B_{\ch_{Q^{\o^2\bar\otimes\pi_{n}^{x}}}\we\cd}\Big\|_\infty +2\e.&
\eeaa
As in Lemma \ref{lem: s.t. uc}, one may show that
\beaa
\Big|\overline{\theta}_{N,n}^{\omega^{1}}-\overline{\theta}_{N,n}^{\omega^{2}}\Big|
& \leq & \rho \Big(d^e(\omega^{1},\omega^{2})+2\e\Big)~\le ~ \rho \big(d^{e}(\omega^{1},\omega^{2})\big) + \rho(2\e),
\eeaa
in particular, $\rho$ is independent of $N$ and $\e$.  By sending $N\rightarrow\infty$, we obtain that
\beaa
\Big|\overline{\theta}_{n}^{\omega^{1}}-\overline{\theta}_{n}^{\omega^{2}}\Big| &\leq & \rho \big(d^{e}(\omega^{1},\omega^{2})\big) + \rho(2\e).
\eeaa
A similar argument provides the same estimate for $\underline{\theta}_{n}^{\omega}$:
\be\label{eq:under theta}
\Big|\underline{\theta}_{n}^{\omega^{1}}-\underline{\theta}_{n}^{\omega^{2}}\Big| ~\leq ~ \rho \big(d^{e}(\omega^{1},\omega^{2})\big) + \rho(2\e).
\ee

\vspace{3mm}
\no \underline{\it Step 3.}\quad We now prove (i) for general $G$. Given the construction
of Step 1, we define:
\beaa
\ol{\th}_{m}^{\o,m}(\pi_{m};x) ~:=~ \ol{\th}_{m}^{\o}(\pi_{m};x),\q  \underline{\theta}_{m}^{\o,m}(\pi_{m};x)~:=~ \underline{\theta}_{m}^{\o}(\pi_{m};x),\q  m\geq1.
\eeaa
For $n\leq m-1$, we define $\overline{\theta}_{n}^{\o,m}$
and $\underline{\theta}_{n}^{\o,m}$ as the unique C-L viscosity
solution of the path-frozen PDE $(E)_{\epsilon}^{\o\bar\otimes\pi_{n}}$ with the boundary conditions
$$
\ol{\th}_n^{\o,m}(\pi_{n};x)=\ol{\th}_{n+1}^{\o,m}(\pi_{n}^{x};0),
~~\ul{\th}_n^{\o,m}(\pi_{n};x)=\ul{\th}_{n+1}^{\o,m}(\pi_{n}^{x};0)~~\mbox{for}~x\in\partial O_\e(\o\bar\otimes\pi_n).
$$
Since $\ul g\leq G\leq \ol g$, it is obvious that $\ol\th^{\e,m}_m$ and $\ul\th^{\e,m}_m$ are respectively C-L viscosity supersolution and subsolution to the path-frozen PDE $(E)_{\epsilon}^{\o\bar\otimes\pi_m}$. By the comparison result for C-L viscosity solutions of PDE's, we obtain that
\beaa
\overline{\theta}_{m}^{\o,m}(\pi_{m};\cdot)
\geq
\overline{\theta}_{m}^{\o,m+1}(\pi_{m};\cdot)
\geq
\underline{\theta}_{m}^{\o,m+1}(\pi_{m};\cdot)\geq\underline{\theta}_{m}^{\o,m}(\pi_{m};\cdot)\q\mbox{on}~O_\e(\o\bar\otimes\pi_m),
\eeaa
Further, it follows from the comparison again that
\begin{equation}\label{eq:comparison theta}
\ol{\th}_n^{\o,m}(\pi_n;\cdot)
\geq
\ol{\th}_n^{\o,m+1}(\pi_n;\cdot)
\geq
\ul{\th}_n^{\o,m+1}(\pi_n;\cdot)
\geq\ul{\th}_n^{\o,m}(\pi_n;\cdot)\q \mbox{on}~O_\e(\o\bar\otimes\pi_n)~\mbox{for all}~n\leq m.
\end{equation}
Denote $\d\th_n^{\o,m}:=\ol\th_n^{\o,m}-\ul\th_n^{\o,m}$.
Applying Lemma \ref{lem: comp on ball} repeatedly and using the
tower property of $\overline{\mathcal{E}}^{L_0}$ stated in Lemma \ref{lem: tower property}, we obtain that
\beaa
|\d\th_n^{\o,m}(\pi_n;x)|\leq\overline{\cE}^{L_0}\Big[\left|\delta\theta_{m}^{\o,m}\left(\pi_{n}^{m-n}(x,B);0\right)\right| 1_{\{\ch^x_{m-n} < \ch_{Q^{\o\bar\otimes \pi^x_n}} \}}\Big]
\eeaa
(we also used the fact that $\d\th_m^{\o,m}(\o' ; 0)=0$ as $\o' \in\partial \cQ^{\o}$).
Then, by Proposition \ref{prop: c1 stopping}, we have
$$
|\delta\theta_{n}^{\o,m}(\pi_{n};x)|
\leq 
C\mathcal{C}^{L_{0}}\left[\ch_{m-n}^{x}<\ch_{Q^{\o\bar\otimes\pi^x_n}}\right]\rightarrow 0,\q\mbox{as}~m\rightarrow\infty.
$$
Together with (\ref{eq:comparison theta}), this implies the existence of $\th_n^\o$ such that
\be\label{eq: limit theta}
\overline{\theta}_{n}^{\o,m}\downarrow\theta_{n}^{\o},\q \underline{\theta}_{n}^{\o,m}\uparrow\theta_{n}^{\o},\q \mbox{as}\ m\rightarrow\infty.
\ee
Clearly $\theta_{n}^\o$ is uniformly bounded and continuous (because it is both lower and upper semicontinuous).
Finally, it follows from the stability of C-L viscosity solutions that $\theta_{n}^\o$ satisfies the statement of (i).

\vspace{3mm}
\no \underline{\it Step 4.}\q We next prove (ii) for a general nonlinearity $G$.  For the simplicity of notation, we denote the stopping times:
\beaa
\ch^{i}:=\ch_{Q^{\omega^{i}\bar\otimes\pi_{n}^{x}}} \q \mbox{for}\q i=1,2,\q \ch^{1,2}:=\ch^{1}\wedge \ch^{2}.
\eeaa
First, considering $\overline{\theta}_{n}^{\omega,m}$ defined in Step 3, we claim that for $\pi_{n}^{x}\in\mathcal{Q}^{\omega^{1}}\cap\mathcal{Q}^{\omega^{2}}$
\begin{multline}\label{eq: crucial}
(\overline{\theta}_{n}^{\omega^{1},m}-\underline{\theta}_{n}^{\omega^{2},m})(\pi_{n};x) 
 \leq 
  \overline{\mathcal{E}}^{L_{0}}\Big[(\ol\th_{m}^{\o^{1}}-\ul\th_{m}^{\o^{2}})\Big(\pi_{n}^{m-n}(x,B);0\Big)1_{\left\{ \ch^x_{m-n}\leq \ch^{1,2}\right\} }  \\
  + \big(\rho \big(d^e(\o^1,\o^2)\big)+\rho(2\e)\big)1_{\left\{ \ch^x_{m-n} > \ch^{1,2}\right\} }  \Big]+ C(m-n)\rho\big(d^{e}(\omega^{1},\omega^{2})\big),
\end{multline}
This claim will be proved in Step 5. 
Since $\ol\th_{m}^{\o^{1}}, \ul\th_{m}^{\o^{2}}$ are both bounded, it follows from (\ref{eq: crucial}) that
\beaa
(\overline{\theta}_{n}^{\omega^{1},m}-\underline{\theta}_{n}^{\omega^{2},m})(\pi_{n};x)
&\le & 
C\mathcal{C}^{L}\left[\ch^x_{m-n}<\ch^{1,2}\right]+C(m-n+1)\rho \big(d^{e}(\omega^{1},\omega^{2})\big) + \rho (2\e).
\eeaa
Recalling \eqref{eq: limit theta}, we obtain that
\beaa
(\theta_{n}^{\omega^{1}}-\theta_{n}^{\omega^{2}})(\pi_{n};x)
&\leq &
 C\mathcal{C}^{L}\left[\ch^x_{m-n}<\ch^{1,2}\right]+C(m-n+1)\rho \big(d^{e}(\omega^{1},\omega^{2})\big) + \rho (2\e).
\eeaa
Since $\lim_{m\rightarrow\infty} \mathcal{C}^{L}\left[\ch^x_{m-n}<\ch^{1,2}\right] =0$, there is a constant $C_\e$ such that
\beaa
(\theta_{n}^{\omega^{1}}-\theta_{n}^{\omega^{2}})(\pi_{n};x)
&\leq &
\e+ C_\e \rho \big(d^{e}(\omega^{1},\omega^{2})\big) + \rho (2\e).
\eeaa
By exchanging the roles of $\omega^{1}$ and $\omega^{2}$, we have 
\beaa
\big|(\theta_{n}^{\omega^{1}}-\theta_{n}^{\omega^{2}})(\pi_{n};x)\big|
& \leq & \e+\rho(2\e) + C_\e \rho \big(d^{e}(\omega^{1},\omega^{2})\big).
\eeaa

\vspace{3mm}

\no \underline{\it Step 5.}\quad We now prove Claim (\ref{eq: crucial}). Suppose that $m\geq n+1$.
We first  show that
\bea\label{complex1}
(\overline{\theta}_{n}^{\omega^{1},m}-\underline{\theta}_{n}^{\omega^{2},m})(\pi_{n};x) & \leq & \overline{\mathcal{E}}^{L_{0}}\Big[ \big(\overline{\theta}_{n+1}^{\omega^{1},m}-\underline{\theta}_{n+1}^{\omega^{2},m}\big) \big(\pi_{n}^{1}(x,B);0\big)1_{\left\{ \ch^x_{1}\leq \ch^{1,2}\right\} }\\
 &  &  + \big(\rho\big(d^e(\o^1,\o^2)\big) + \rho(2\e)\big) 1_{\{\ch^x_1 > \ch^{1,2}\}}  \Big] +C\rho(d^{e}(\omega^{1},\omega^{2})).\notag
\eea
Then (\ref{eq: crucial}) follows from simple induction. Recall
that $\overline{\theta}_{n}^{\omega^{1},m}$ (resp. $\underline{\theta}_{n}^{\omega^{2},m}$)
is a solution to the PDE with generator $G(\omega^{1},\cdot)$ (resp. $G(\omega^{2},\cdot)$).
Now we study those two PDE's on the domain:
\[
O_{\epsilon}\cap Q^{\omega^{1}}\cap Q^{\omega^{2}}.
\]
The boundary of this set can be divided into three parts which belong
to $\partial O_{\epsilon}$, $\partial Q^{\omega^{1}}$ and $\partial Q^{\omega^{2}}$
respectively. We denote them by ${\rm Bd}_{1}$, ${\rm Bd}_{2}$ and ${\rm Bd}_{3}$.

(i)\quad On ${\rm Bd}_{1}$, we have $\ch^x_{1}\leq \ch^{1,2}$, and thus 
$$\ol{\th}_n^{\o^1,m}(\pi_n;x)=\ol{\th}_{n+1}^{\o^1,m}(\pi^x_n;0)~~\mbox{and}~~\ul{\th}_n^{\o^2,m}(\pi_n;x)=\ul{\th}_{n+1}^{\o^2,m}(\pi^x_n,0).$$

(ii)\quad On ${\rm Bd}_{2}$, we have $\ch^{1}<\ch^x_{1}$, so we  have $~\ol{\th}_n^{\o^1,m}(\pi_n;x)= \xi(\o^1\bar\otimes \pi^x_n) = \ul{\th}_{n}^{\o^1,n}(\pi_n;x).$

(iii)\quad On ${\rm Bd}_{3}$, we have $\ch^{2}<\ch^x_{1}$, so we have $~\ul{\th}_n^{\o^2,m}(\pi_n;x)=\xi(\o^2\bar\otimes \pi^x_n)=\ol{\th}_{n}^{\o^2,n}(\pi_n;x).$

\no Then it follows from Lemma \ref{lem: comp on ball} that
\bea\label{complex2}
(\ol{\th}_{n}^{\o^{1},m}-\ul{\th}_{n}^{\o^{2},m})(\pi_{n};x) 
& \leq & \ol{\cE}^{L_{0}}\Big[\big(\overline{\theta}_{n+1}^{\omega^{1},m} - \ul{\th}_{n+1}^{\o^{2},m}\big)\big(\pi_{n}^{1}(x,B);0\big)1_{\left\{ \ch^x_{1}\leq \ch^{1,2}\right\} }\notag \\
 &  &  + \Big(\ul{\th}_{n}^{\o^1,n}(\pi_n; x+B_{\ch^1}) - \ul{\th}_{n}^{\o^{2},m}\big(\pi_{n}; x+B_{\ch^1}\big) \Big) 1_{\{\ch^1<\ch^x_1 \le \ch^{2}\}}  \\
 & &   + \Big(\ol{\th}_{n}^{\o^{1},m}\big(\pi_{n}; x+ B_{\ch^2}\big) - \ol{\th}_{n}^{\o^2,n}(\pi_n;x+B_{\ch^2})  \Big) 1_{\{\ch^2<\ch^x_1 \le \ch^{1}\}}  \Big] +C\rho(d^{e}(\omega^{1},\omega^{2})). \notag
\eea
We next estimate
\beaa
\D &:=& \ul{\th}_{n}^{\o^1,n}(\pi_n; x+B_{\ch^1}) - \ul{\th}_{n}^{\o^{2},m}\big(\pi_{n}; x+B_{\ch^1}\big)
\eeaa
As in Step 3, the comparison result of C-L viscosity solution implies that
\beaa
\ul{\th}_{n}^{\o^{2},m}(\pi_{n};x+B_{\ch^{1}})
&\geq &
\ul{\th}_{n}^{\o^{2},n}(\pi_{n};x+B_{\ch^{1}}).
\eeaa
It follows from \eqref{eq:under theta} that
\beaa
\D 
~\le~
\ul\th_{n}^{\o^1,n}(\pi_n; x+ B_{\ch^1})-\ul\th_{n}^{\o^2,n}(\pi_n; x+B_{\ch^1})
~\leq~ 
\rho \big( d^e(\omega^1,\omega^2)\big)+\rho(2\e).
\eeaa
Similarly we can obtain the same estimate for $ \ol{\th}_{n}^{\o^{1},m}\big(\pi_{n}; x+B_{\ch^2}\big)- \ol{\th}_{n}^{\o^2,n}(\pi_n; x+B_{\ch^2}) $. Together with \eqref{complex2}, we obtain \eqref{complex1}.
\qed

\vspace{3mm}
The previous lemma shows the existence of C-L viscosity solution to
the path-frozen PDE's. Further, we will use Assumption \ref{assmp: approx PDE} to construct piecewise smooth super- and sub-solutions to the PPDE.
Recall the stopping times defined in \eqref{def:seq_stop}, and denote 
$$\th^\e_n:=\th^{{\bf 0},\e}_n,\q 
\ch_n:=\ch^0_n\we \ch_Q \q\mbox{and}\q \hat\pi_n := {\rm Lin}\big\{(\ch_{i}(\o),\o_{\ch_{i}(\o)});0\leq i\leq n\big\}. $$ 

\begin{lem}\label{lem: construct psi}
There exists $\psi^{\epsilon}\in\overline{C}^{2}(\mathcal{Q})$
such that
\beaa
&\psi^{\epsilon}({\bf 0})=\theta_{0}^{\epsilon}({\bf 0})+\epsilon,\q \psi^{\epsilon}\geq \xi ~~ \text{on}~~ \mathcal{\pa \cQ},&\\
&-G\big(\hat\pi_n, \psi^\e(\o), \pa_\o \psi^\e(\o), \pa_{\o\o}^2  \psi^\e(\o) \big)\geq 0\q
 \text{when}~~  \ch_{n}(\o)\le \bar t(\o)<\ch_{n+1}(\o) ,\q \mbox{for all}~n\in\dbN,&
\eeaa
where $ \pa_\o \psi^\e, \pa_{\o\o}^2  \psi^\e$ are the derivatives of $\f^\e$ on the corresponding intervals.
\end{lem}
\begin{proof}
For simplicity, in the proof we omit the superscript $\epsilon$. First, since PDE $(E)_{\epsilon}^{0}$ satisfies  Assumption \ref{assmp: approx PDE} and $G(\o,y,z,\g)$ is decreasing in $y$, there exists a function $v_{0}\in C^2_0(O_\e({\bf 0}))$ such that
\beaa
v_{0}(0)=\theta_{0}(0)+\frac{\epsilon}{2},~~{\bf L}^{0}v_{0}\geq0\ \text{on}\ O_\e({\bf 0})~~\mbox{and}~ v_{0}\geq\theta_{0}\ \text{on}\ \partial O_\e({\bf 0}).
\eeaa
Denote $v_0({\bf 0}; \cd):= v_0(\cd)$. Similarly, applying Assumption \ref{assmp: approx PDE} to PDE $(E)_{\epsilon}^{\hat\pi_{n}}$ ($n\ge 1$),
we can find a function $v_{n}(\hat\pi_{n};\cdot)\in C_0^{2}(O_\e(\hat\pi_n))$
such that
\beaa
& v_{n}(\hat\pi_{n};0)=v_{n-1}\big(\hat\pi_{n-1}; \o_{\ch_{n}(\o)}-\o_{\ch_{n-1}(\o)}\big) + 2^{-n-1}\e,&\\
&{\bf L}^{\hat\pi_{n}} v_{n} (\hat\pi_n;\cd)
\geq 0 ~~ \text{on}~O_\e(\hat\pi_n),\q v_{n}(\hat\pi_{n};\cdot)\geq\theta_{n}(\hat\pi_{n};\cdot) ~~\text{on}~\partial O_\e(\hat\pi_{n}).&
\eeaa
We now give the definition of the required function $\psi: \cQ\rightarrow \dbR$:
\beaa
\psi(\omega)
&:=&
\sum_{n=0}^\infty \Big( v_{n}\big(\hat\pi_n; \omega_{\bar{t}(\o)} - \o_{\ch_n(\o)} \big)+ \e- 2^{-n-1}\e\Big) 1_{\{\ch_{n}(\o)\le \bar t(\o)<\ch_{n+1}(\o)\}}.
\eeaa
Clearly, we have $\psi\in \ol C^2(\cQ)$. Consider a path $\o$ such that $\ch_{n}(\o)\le \bar t(\o)<\ch_{n+1}(\o)$. Since $\psi(\o) \ge v_n\big(\hat\pi_n; \o_{\bar t(\o)} - \o_{\ch_n(\o)}\big)$, it follows from the monotonicity of $G$
\beaa
-G\big(\hat\pi_n, \psi(\o), \pa_\o \psi(\o), \pa_{\o\o}^2  \psi(\o) \big)
 ~\ge ~ {\bf L}^{\hat\pi_n} v_n\big(\hat\pi_n; \o_{\bar t(\o)} - \o_{\ch_n(\o)}\big)
 ~\ge ~ 0.
\eeaa
Finally, we may easily check that
$\psi(0)-\theta_{0}(0)=\frac{\epsilon}{2}+\frac{\e}{2}=\epsilon$, and that $\psi \ge \xi$ on $\pa\cQ$.
\end{proof}

\no Now we have done all the necessary constructions and are ready
to show the main result of the section.

\vspace{5mm}

\no {\bf Proof of Proposition \ref{prop:u=u}}\quad
For any $\epsilon>0$, let $\psi^{\epsilon}$ be as in Lemma \ref{lem: construct psi},
and $\overline{\psi}^{\epsilon}:=\psi^{\epsilon}+\rho(2\epsilon)+\lambda^{-1}\left(\rho(2\epsilon)\right)$,
where $\rho$ is the common modulus of continuity of $\xi$ and $G$, and
$\lambda^{-1}$ is the inverse of the function in Assumption \ref{assump: generator}.
Then clearly $\overline{\psi}^{\epsilon}\in\overline{C}^{2}(\mathcal{Q})$
and bounded. Also,
\[
\overline{\psi}^{\epsilon}(\omega)-\xi(\omega)\geq\psi^{\epsilon}(\omega)+\rho(2\epsilon)-\xi(\omega)\geq\xi(\omega^{\epsilon})-\xi(\omega)+\rho(2\epsilon)\geq0\ \text{on}\ \partial\mathcal{Q}.
\]
Moreover, when $\bar{t}(\omega)\in[H_{n}(\omega),H_{n+1}(\omega))$,
we have that
\begin{eqnarray*}
\mathcal{L}\overline{\psi}^{\epsilon}(\omega) & = & -G\left(\omega,\overline{\psi}^{\epsilon},\partial_{\omega}\psi^{\epsilon},\partial_{\omega\omega}^{2}\psi^{\epsilon}\right)\\
 & \geq & -G\left(\hat{\pi}_{n},\psi^{\epsilon}+\lambda^{-1}\left(\rho(2\epsilon)\right),\partial_{\omega}\psi^{\epsilon},\partial_{\omega\omega}^{2}\psi^{\epsilon}\right)-\rho(2\epsilon)\\
 & \geq & -G\left(\hat{\pi}_{n},\psi^{\epsilon},\partial_{\omega}\psi^{\epsilon},\partial_{\omega\omega}^{2}\psi^{\epsilon}\right)\geq0.
\end{eqnarray*}
Then by the definition of $\overline{u}$ we see that
\be\label{estim: u and theta}
\overline{u}(0)\leq\overline{\psi}^{\epsilon}(0)=\psi^{\epsilon}+\rho(2\epsilon)+\lambda^{-1}\left(\rho(2\epsilon)\right)\leq\theta_{0}^{\epsilon}(0)+\epsilon+\rho(2\epsilon)+\lambda^{-1}\left(\rho(2\epsilon)\right).
\ee
Similarly, $\underline{u}(0)\geq\theta_{0}^{\epsilon}(0)-\epsilon-\rho(2\epsilon)-\lambda^{-1}\left(\rho(2\epsilon)\right)$.
That implies that
\[
\overline{u}(0)-\underline{u}(0)\leq2\epsilon+2\rho(2\epsilon)+2\lambda^{-1}\left(\rho(2\epsilon)\right).
\]
Since $\epsilon$ is arbitrary, this shows that $\overline{u}(0)=\underline{u}(0)$.
Similarly, we can show that $\overline{u}(\omega)=\underline{u}(\omega)$
for all $\omega\in\mathcal{Q}$.
\qed

\section{Existence}\label{sec: Existence}

In this section, we verify that 
\be\label{defn: u}
u~:=~\ol u~=~\ul u
\ee
is the unique $\cP^L$-viscosity solution in ${\rm BUC}(\cQ)$ to the PPDE (\ref{PPDE}) for $L\ge L_0$. We will prove that $u \in {\rm BUC}(\cQ)$ in Subsection {\ref{subsec: regularity} and $u$ satisfies the viscosity property in Subsection {\ref{subsec: visco}}. 

\subsection{Regularity}\label{subsec: regularity}

The non-continuity of the hitting time $\ch_Q(\cd)$ brings difficulty to the proof of the regularity of $u$. One cannot adapt the method used in \cite{ETZ3}. In our approach, we make use of the estimate \eqref{eq:theta uni} for the solution of the path-frozen PDE's.

\begin{prop}\label{prop: u bounded}
Let Assumption \ref{assump: generator} hold and $\xi\in\mbox{\rm BUC}(\pa \cQ)$.
Then $\overline{u}$ is bounded from above and $\underline{u}$ is
bounded from below. 
\end{prop}
\begin{proof}
Assume that $|\xi|\leq C_0$. Define:
\[
\psi:=\lambda^{-1}\left(C_{0}\right)+C_{0}.
\]
Obviously $\psi\in\bar{C}^{2}$. Observe that $\psi_{T}\geq C_{0}\geq\xi$.
Also,
\[
\mathcal{L}^\o\psi_{s}=-G^{\omega}(\cdot,\psi_{s},0,0)\geq C_{0}-G^{\omega}(\cdot,0,0,0)\geq0.
\]
It follows that $\psi\in\overline{\mathcal{D}}_{Q}^{\xi}(\omega)$,
and thus $\overline{u}(\omega)\leq\psi(0)=\lambda^{-1}\left(C_{0}\right)+C_{0}$.
Similarly, one can show that $\underline{u}(\omega)\geq-\lambda^{-1}\left(C_{0}\right)-C_{0}$.
\end{proof}

\begin{prop}\label{prop: UC}
The function $u$ defined in \reff{defn: u} is uniformly continuous in $\mathcal{Q}$.
\end{prop}

\proof Recall \eqref{estim: u and theta}, i.e. for $\omega^{1},\omega^{2}\in\mathcal{Q}$, it holds that
\beaa
\overline{u}(\omega^{1})\leq\theta_{0}^{\omega^{1}}(0)+\epsilon+\rho(2\epsilon)
\q \mbox{and} \q 
\underline{u}(\omega^{2})\geq\theta_{0}^{\omega^{2}}(0)-\epsilon-\rho(2\epsilon).
\eeaa
Hence, it follows from Lemma \ref{lem:frozen construct} that
\begin{multline*}
u(\omega^{1})-u(\omega^{2})~=~\overline{u}(\omega^{1})-\underline{u}(\omega^{2})\\
~\leq~\theta_{0}^{\omega^{1}}(0)-\theta_{0}^{\omega^{2}}(0)+2(\epsilon+\rho(2\epsilon))
~\leq~ C_\e \rho(d^{e}(\omega^{1},\omega^{2}))+3\big(\e+\rho(2\epsilon)\big),
\q \mbox{for all}\q\e>0.
\end{multline*}
By exchanging the roles of $\omega^{1}$ and $\omega^{2}$, we obtain
$|u(\omega^{1})-u(\omega^{2})| \le C_\e \rho(d^{e}(\omega^{1},\omega^{2}))+3\big(\e+\rho(2\epsilon)\big)$, from which the uniform continuity of $u$ can be easily deduced.
\qed

\subsection{Viscosity property}\label{subsec: visco}

After having shown that $u$ is uniformly
continuous, we need to verify that it indeed satisfies the viscosity property. The following proof is similar to that of Proposition 4.3 in \cite{ETZ3}.
\begin{prop}
The function $u$ defined in \eqref{defn: u} is a $\cP^L$-viscosity solution to PPDE (\ref{PPDE}) for $L \ge L_0$.
\end{prop}
\begin{proof}
We only prove that $\overline{u}$ is
a $\cP^L$-viscosity supersolution. The subsolution property can be proved similarly. Without loss of generality, we only show the $\cP^{L_0}$-viscosity supersolution property at the point ${\bf 0}$. Assume the contrary, i.e.
there exists $\varphi\in\overline{\mathcal{A}}^{L_{0}}\overline{u}({\bf 0})$
such that $-c:=\mathcal{L}\varphi({\bf 0})<0$. For any $\psi\in\overline{\mathcal{D}}_{Q}^{\xi}({\bf 0})$
and $\omega\in\mathcal{Q}$ it is clear that $\psi^{\omega}\in\overline{\mathcal{D}}_{Q}^{\xi}(\omega)$
and $\psi(\omega)\geq\overline{u}(\omega)$. Now by the definition
of $\overline{u}$, there exists $\psi^{n}\in\overline{C}^{2}(\mathcal{Q})$
such that
\begin{equation}
\delta_{n}:=\psi^{n}(0)-\overline{u}(0)\downarrow0\text{ as }n\rightarrow\infty,\ \mathcal{L}\psi^{n}(\omega)\geq0,\ \omega\in\mathcal{Q}.\label{eq:psi n}
\end{equation}
Let $\ch_\e:=\e\we\ch_{O_\e}$ be a localization of test function $\f$.
Since $\varphi\in C^{2}(\cO_\e)$ and $\overline{u}\in\mathrm{BUC}(\mathcal{Q})$,
without loss of generality we may assume that
\be\label{phipetit}
\mathcal{L}\varphi(\omega_{t\wedge\cdot})\leq-\frac{c}{2}\text{ and }|\varphi_{t}-\varphi_{0}|+\left|\overline{u}_{t}-\overline{u}_{0}\right|\leq\frac{c}{6L_{0}}\text{ for all }t\leq \ch_{O_\e}.
\ee
Since $\varphi\in\mathcal{\overline{A}}^{L_{0}}\overline{u}({\bf 0})$,
this implies for all $\mathbb{P}\in\mathcal{P}^{L_{0}}$ that :
\begin{equation}
0\geq\mathbb{E}^{\mathbb{P}}\left[(\varphi-\overline{u})_{\ch_{\e}}\right]\geq\mathbb{E}^{\mathbb{P}}\left[(\varphi-\psi^{n})_{\ch_{\e}}\right].\label{eq:ep}
\end{equation}
Denote $\mathcal{G}^{\mathbb{P}}\phi:=\alpha^{\mathbb{P}}\cdot\partial_{\omega}\phi+\frac{1}{2}(\beta^{\mathbb{P}})^{2}:\partial_{\omega\omega}^{2}\phi$.
Then, since $\varphi\in C^{2}(\cO_\e)$ and $\psi^{n}\in\ol C^2(\cQ)$, it follows from (\ref{eq:psi n}) that:
\begin{eqnarray*}
\delta_{n} 
& \geq &
 \mathbb{E}^{\mathbb{P}}\left[(\varphi-\psi^{n})_{\ch_{\e}}-(\varphi-\psi^{n})_{0}\right]
 =
 \mathbb{E}^{\mathbb{P}}\left[\int_{0}^{\ch_{\e}}\mathcal{G}^{\mathbb{P}}(\varphi-\psi^{n})(B_{s\we\cd})ds\right]\\
 & \geq &
  \mathbb{E}^{\mathbb{P}}\biggl[\int_{0}^{\ch_{\e}} \Big(\frac{c}{2}-G(\cd,\varphi,\partial_{\omega}\varphi,\partial_{\omega\omega}^{2}\varphi)+G(\cd,\psi^{n},\partial_{\omega}\psi^{n},\partial_{\omega\omega}^{2}\psi^{n})+\mathcal{G}^{\mathbb{P}}(\varphi-\psi^{n})\Big) (B_{s\wedge\cdot})ds\biggr]\\
 & \ge & 
 \mathbb{E}^{\mathbb{P}}\biggl[\int_{0}^{\ch_{\e}} \Big(\frac{c}{2}-G(\cd,\varphi,\partial_{\omega}\varphi,\partial_{\omega\omega}^{2}\varphi) +G(\cd,\overline{u},\partial_{\omega}\psi^{n},\partial_{\omega\omega}^{2}\psi^{n})+\mathcal{G}^{\mathbb{P}}(\varphi-\psi^{n})\Big)(B_{s\wedge\cdot})ds\biggr],
\end{eqnarray*}
where the last inequality is due to the monotonicity in $y$ of $G$. Since $\varphi_{0}=\overline{u}_{0}$ and $G$ is $L_0$-Lipschitz continuous in $y$, it follows from \eqref{phipetit} that
\begin{eqnarray*}
\delta_{n} 
& \geq &
 \mathbb{E}^{\mathbb{P}}\biggl[\int_{0}^{\ch_{\e}}\Big(\frac{c}{3}-G(\cd,\overline{u}_{0},\partial_{\omega}\varphi,\partial_{\omega\omega}^{2}\varphi)+G(\cd,\overline{u}_{0},\partial_{\omega}\psi^{n},\partial_{\omega\omega}^{2}\psi^{n})+\mathcal{G}^{\mathbb{P}}(\varphi-\psi^{n})\Big) (B_{s\wedge\cdot})ds\biggr].
\end{eqnarray*}
We next let $\eta>0$, and for each $n$, define $\tau_{0}^{n}:=0$
and
\begin{eqnarray*}
\tau_{j+1}^{n}(\o): & = & \ch_{\e}(\o) \wedge\inf\{t\geq\tau_{j}^{n}:\rho(d^{e}(\omega_{t\wedge\cdot},\omega_{\tau_{j}^{n}\wedge\cdot}))+|\partial_{\omega}\varphi(\omega_{t\wedge\cdot})-\partial_{\omega}\varphi(\omega_{\tau_{j}^{n}\wedge\cdot})|\\
 &  & +|\partial_{\omega\omega}^{2}\varphi(\omega_{t\wedge\cdot})-\partial_{\omega\omega}^{2}\varphi(\omega_{\tau_{j}^{n}\wedge\cdot})|+|\partial_{\omega}\psi^{n}(\omega_{t\wedge\cdot})-\partial_{\omega}\psi^{n}(\omega_{\tau_{j}^{n}\wedge\cdot})|\\
 &  & +|\partial_{\omega\omega}^{2}\psi^{n}(\omega_{t\wedge\cdot})-\partial_{\omega\omega}^{2}\psi^{n}(\omega_{\tau_{j}^{n}\wedge\cdot})|\geq\eta\},
\end{eqnarray*}
where $\rho$ is a modulus of continuity in $\o$ of $G$.
Since $\f\in C^2(\cO_\e)$ and $\psi^n\in \ol C^2(\cQ)$, one can easily check that $\tau_{j}^{n}\uparrow \ch_{\e}$, $\mathcal{P}^{L_{0}}$-q.s. as $j\rightarrow\infty$. Thus,
\begin{eqnarray*}
\delta_{n} 
& \geq &
 \Big(\frac{c}{3}-C\eta\Big)\mathbb{E}^{\mathbb{P}}[\ch_{\e}]+\sum_{j\geq0}\mathbb{E}^{\mathbb{P}} \big(\tau_{j}^{n}-\tau_{j+1}^{n}\big)\Bigl(-G(\cdot,\overline{u}_{0},\partial_{\omega}\varphi,\partial_{\omega\omega}^{2}\varphi)\\
 &  & \q\q\q\q\q\q\q\q\q\q\q\q\q\q\q\q\q\q +G(\cdot,\overline{u}_{0},\partial_{\omega}\psi^{n},\partial_{\omega\omega}^{2}\psi^{n})+\mathcal{G}^{\mathbb{P}}(\varphi-\psi^{n})\Bigr)(B_{\tau_{j}^{n}\we\cd})\\
 & = & \Big(\frac{c}{3}-C\eta\Big)  \mathbb{E}^{\mathbb{P}}[\ch_{\e}]+\sum_{j\geq0}\mathbb{E}^{\mathbb{P}}\big(\tau_{j}^{n}-\tau_{j+1}^{n}\big) \Bigl(\alpha_{j}^{n}\cdot\partial_{\omega}(\psi^{n}-\varphi)
 +\frac{1}{2}(\beta_{j}^{n})^{2}:\partial_{\omega\omega}^{2}(\psi^{n}-\varphi)+\mathcal{G}^{\mathbb{P}}(\varphi-\psi^{n})\Bigr)(B_{\tau_{j}^{n}\we\cd}),
\end{eqnarray*}
for some $\alpha_{j}^{n}$, $\beta_{j}^{n}$ such that $|a^n_j|\le L$ and $\b^n_j\in\dbH^0_L$. Note that
$\alpha_{j}^{n}$ and $\beta_{j}^{n}$ are both $\mathcal{F}_{\tau_{j}^{n}}$-measurable. Take $\mathbb{P}_{n}\in\mathcal{P}^{L_{0}}$ such that $\alpha_{t}^{\mathbb{P}_{n}}=\alpha_{j}^{n}$,
$\beta_{t}^{\mathbb{P}_{n}}=\beta_{j}^{n}$ for $t\in[\tau_{j}^{n},\tau_{j+1}^{n})$.
Then
\beaa
\delta_{n} &\geq & \Big(\frac{c}{3}-C\eta\Big)\mathbb{E}^{\mathbb{P}_{n}}[\ch_{\e}].
\eeaa
Let $\eta:=\frac{c}{6C}$. It follows that $\underline{\mathcal{E}}^{L_{0}}\left[\ch_{\e}\right] \leq \mathbb{E}^{\mathbb{P}_{n}}\left[\ch_{\e}\right] \leq \frac{6}{c}\delta_{n}$.
By letting $n\rightarrow\infty$, we get $\underline{\mathcal{E}}^{L_{0}}\left[\ch_{\e}\right]=0$, contradiction.
\end{proof}

\section{Path-dependent time-invariant stochastic control}\label{sec: app}

In this section, we present an application of fully nonlinear elliptic PPDE.  An important question which is most relevant since the recent financial crisis is the risk of model mis-specification. The uncertain volatility model (see  Avellaneda, Levy and Paras \cite{ALP}, Lyons \cite{Lyons} or Nutz \cite{N-G}) provides a conservative answer to this problem.  

In the present application, the canonical process $B$ represents the price process of some primitive asset, and our objective is the hedging of the derivative security defined by the payoff $\xi(B_\cdot)$ at some maturity $\ch_Q$ defined as the exiting time from some domain $Q$.

In contrast with the standard Black-Scholes modeling, we assume that the probability space $(\O,\cF)$ is endowed with a family of probability measures $\cP^{\rm UVM}$. In the uncertain volatility model, the quadratic variation of the canonical process is assumed to lie between two given bounds,
$$\ul\si^2dt\le d\langle B\rangle_t\le \ol\si^2dt,\q 
\dbP\mbox{-a.s. for all}~\dbP\in\cP^{\rm UVM}.
$$ 
Then, by the possible frictionless trading of the underlying asset, it is well known that the non-arbitrage condition is characterized by the existence of an equivalent martingale measure. Consequently, we take
\beaa
\cP^{\rm UVM}:=\{\dbP\in \cP^{\infty}: B~\mbox{is a continuous}~\dbP\mbox{-martingale  and}~\frac{d\langle B\rangle_t}{dt}\in [\ul\si^2,\ol\si^2],~\dbP\mbox{-a.s.}\}.
\eeaa 
The superhedging problem under model uncertainty was initially formulated by Denis \& Martini \cite{DM} and Neufeld \& Nutz \cite{Nutz}, and involves delicate quasi-sure analysis. Their main result expresses the cost of robust superheging as 
$$u_0:=\ol\cE^{\rm UVM}\big[e^{-r\ch_Q}\xi(B_{\ch_Q\we\cd})\big]
:=\ol\cE^{\cP^{\rm UVM}}\big[e^{-r\ch_Q}\xi(B_{\ch_Q\we\cd})\big],$$
where $r$ is the discount rate. Further, define $u$ on $\O^e$ as:
\be\label{stochcontol}
u(\o) := \ol\cE^{\rm UVM}\big[e^{-r \ch_{Q^\o}} \xi\big(\o\bar\otimes B_{\ch_{Q^\o}\we\cdot}\big)\big],
\q\mbox{for all}~\o\in\cQ.
\ee
We are interested in characterizing $u$ as a viscosity solution of the corresponding fully nonlinear elliptic PPDE.

\begin{assum}\label{assum: app}
Assume that 
\beaa
\xi\in {\rm BUC}(\pa \cQ),\q \ul\si>0,\q \mbox{and the discount rate}~ r\ge0.
\eeaa
\end{assum}

\begin{prop}\label{prop: application}
Let $L$ be a constant such that $\frac{1}{L}\le \ul\si$ and $L\ge \ol\si$. Under Assumption \ref{assum: app}, the function $u$ defined in \eqref{stochcontol} is in ${\rm BUC}(\cQ)$ and is a $\cP^{L}$-viscosity solution to the elliptic path-dependent HJB equation:
\beaa
ru-\sup_{\g \in[\ul\si, \ol\si]}\frac12 \g^2 \pa^2_{\o\o}u~=~0~\mbox{on}~\cQ,~~\mbox{and}~~u=\xi~\mbox{on}~\pa\cQ
\eeaa
\end{prop}

\begin{lem}
The function $u$ defined in \eqref{stochcontol} is in ${\rm BUC}(\cQ)$.
\end{lem}
\begin{proof}
As in Lemma \ref{lem: s.t. uc}, the required result follows easily from the fact $\xi\in\mbox{BUC}(\pa\cQ)$.
\end{proof}

\begin{lem}\label{lem:tower-uvm}
We have $u_0 = \ol\cE^{\rm UVM} [e^{-r\t} u_\t]$ (recall that $u_t(\o):=u(\o_{t\we\cd})$) for all $\t\le \ch_Q$.
\end{lem}
\begin{proof}
By the definition of $u$, we have
\beaa
e^{-rt} u(\o_{t\we\cd}) 
&=& e^{-rt}\ol\cE^{\rm UVM}\Big[e^{-r \ch_{Q^{\o_{t\we\cd}}}} \xi\big(\o \otimes_t B_{\ch_{Q^{\o_{t\we\cd}}}\we\cdot}\big)\Big]\\
&=& e^{-rt}\ol\cE^{\rm UVM}\Big[e^{-r \big((\ch_{Q})^{t,\o}-t\big)} \big(\xi_{\ch_{Q}}\big)^{t,\o}\Big]\\
&=& \ol\cE^{\rm UVM}\Big[e^{-r (\ch_{Q})^{t,\o}} \big(\xi_{\ch_{Q}}\big)^{t,\o}\Big].
\eeaa
Then it follows the tower property (Lemma \ref{lem: tower property}) that
\beaa
u_0 = \ol\cE^{\rm UVM}\Big[e^{-r\ch_Q}\xi(B_{\ch_Q\we\cd})\Big]
 =  \ol\cE^{\rm UVM}\Big[  \ol\cE^{\rm UVM}\big[e^{-r (\ch_{Q})^{\t,\cd}} \big(\xi_{\ch_{Q}}\big)^{\t,\cd}\big] \Big] 
=  \ol\cE^{\rm UVM} [e^{-r\t} u_\t].
\eeaa
\end{proof}

\no {\bf Proof of Proposition \ref{prop: application}}\quad \underline{\it Step 1.}\q We first verify the viscosity supersolution property.  Without loss of generality, we only verify it at the point ${\bf 0}$. Recall the equivalent definition of viscosity solutions in Proposition \ref{prop:equiv-def}. Let $(\a,\b)\in\ol\cJ^L u({\bf 0})$, i.e.
$-u_0=\max_\t\ol\cE^L \big[(\psi^{\a,\b}-u)_{\ch_{\e}\we\t}\big]$, with $\ch_\e:=\e\we\ch_{O_\e}$. 
Then we have for all $\dbP\in\cP^{\rm UVM}\subset \cP^L$ and $h>0$ that
\beaa
0 ~ &\geq & ~ \dbE^{\dbP}\Big[\psi^{\a,\b}_{\ch_{\e}\we h}-u_{\ch_{\e}\we h}+u_0\Big] \\
&\geq & ~ \dbE^\dbP\Big[ \frac12 \b \langle B\rangle_{\ch_{\e}\we h} + \a B_{\ch_{\e}\we h} \Big]+\dbE^\dbP\Big[(e^{-r(\ch_{\e}\we h)}-1)u_{\ch_{\e}\we h}\Big]-\dbE^\dbP\Big[e^{-r(\ch_{\e}\we h)}u_{\ch_{\e}\we h}\Big]+u_0
\eeaa
It follows from Lemma \ref{lem:tower-uvm} that $u_0 = \ol\cE^{\rm UVM}\Big[e^{-r(\ch_{\e}\we h)}u_{\ch_{\e}\we h}\Big] \ge \dbE^\dbP\Big[e^{-r(\ch_{\e}\we h)}u_{\ch_{\e}\we h}\Big]$. Therefore
\beaa
0 &\geq &~ \dbE^\dbP\Big[\frac12 \b \langle B\rangle_{\ch_{\e}\we h}  +\a B_{\ch_{\e}\we h}\Big]+ \dbE^\dbP\Big[(e^{-r(\ch_{\e}\we h)}-1)u_{\ch_{\e}\we h}\Big].
\eeaa
Now, we take $\dbP_{\g}\in\cP^{\rm UVM}$ such that there exists a $\dbP_{\g}$-Brownian motion $W$ such that $B_t = \g W_t$, $\dbP_{\g}$-a.s. It follows that
\beaa
 0 ~\geq~ \frac{1}{h}\dbE^{\dbP_{\g}}\Big[ \frac12 \g^2 \b (\ch_{\e}\we h)+\Big(e^{-r(\ch_{\e}\we h)}-1\Big)u_{\ch_{\e}\we h}\Big].
\eeaa
Let $h\rightarrow 0$, we obtain that $0\geq -ru_0+\frac12 \g^2 \b$. Since $\g \in [\ul\si,\ol\si]$ can be arbitrary, we finally have
\beaa
ru_0-\sup_{\g \in [\ul\si,\ol\si]}\frac12 \g^2 \b ~\geq ~0.
\eeaa

\vspace{3mm}

\no \underline{\it Step 2.}\quad Now we verify the viscosity subsolution property. Without loss of generality, we only verity it at the point ${\bf 0}$. Let $(\a,\b) \in \ul\cJ^L u({\bf 0})$, i.e. $ -u_0=\min_\t\ul\cE^L \big[(\psi^{\a,\b}-u)_{\ch_{\e}\we\t}\big]$, with $\ch_\e:=\e\we\ch_{O_\e}$. For any $h>0$ we have
\beaa
0 ~\le ~\ul \cE^L \Big[\psi^{\a,\b}_{\ch_{\e}\we h}-u_{\ch_{\e}\we h}+u_0\Big].
\eeaa
So we have for all $\dbP\in \cP^{\rm UVM}\subset\cP^L$ that
\beaa
0  &\leq &  \dbE^\dbP\Big[ \frac12 \b \langle B\rangle_{\ch_{\e}\we h} \Big]+\dbE^\dbP\Big[(e^{-r(\ch_{\e}\we h)}-1)u_{\ch_{\e}\we h}\Big]-\dbE^\dbP\Big[e^{-r(\ch_{\e}\we h)}u_{\ch_{\e}\we h}\Big]+u_0\\
  &\le &  \dbE^{\dbP}\Big[\frac12   \sup_{\g \in [\ul\si,\ol\si]}\g^2 \b (\ch_{\e}\we h)  + (e^{-r(\ch_{\e}\we h)}-1)u_{\ch_{\e}\we h}\Big]-\dbE^\dbP\Big[e^{-r(\ch_{\e}\we h)}u_{\ch_{\e}\we h}\Big]+u_0.
\eeaa
Since $u_0 = \ol\cE^{\rm UVM}\Big[e^{-r(\ch_{\e}\we h)}u_{\ch_{\e}\we h}\Big]$ (Lemma \ref{lem:tower-uvm}), it follows that
\be\label{eq:uvmsub}
0 ~\le ~ \ol\cE^{\rm UVM}\Big[\frac12 \sup_{\g \in [\ul\si,\ol\si]}\g^2 \b (\ch_{\e}\we h)  + (e^{-r(\ch_{\e}\we h)}-1)u_{\ch_{\e}\we h}\Big].
\ee
Since we have
\beaa
\Big| \frac{e^{-r(\ch_{\e}\we h)}-1} {h}u_{\ch_{\e}\we h}  + r u_0\Big|
&\le & \Big| \frac{e^{-r(\ch_{\e}\we h)}-1} {h}  + r\Big| |u_{\ch_{\e}\we h}| + r|u_{\ch_{\e}\we h}-u_0|\\
&\le & C \Big| \frac{e^{-r(\ch_{\e}\we h)}-1} {h} + r\Big| + r \rho(\e).
\eeaa
where $\rho$ is a modulus of continuity of $u$. By denoting 
\beaa
\d(h) ~:=~ \sup_{0\le s\le h}\Big| \frac{e^{-rs}-1}{s}+r \Big|,
\eeaa
we have the following estimate:
\beaa
\Big| \frac{e^{-r(\ch_{\e}\we h)}-1} {h}u_{\ch_{\e}\we h} + r u_0\Big|
~\le ~ \Big(C\d (h) +r\rho(\e) \Big) 1_{\{\ch_{\e}>h\}}
+  \Big(C\big(r+\d (h)\big) +r\rho(\e) \Big) 1_{\{\ch_{\e}\le h\}}.
\eeaa
Together with \eqref{eq:uvmsub}, we obtain that
\beaa
0 &\le & \ol\cE^{\rm UVM}\Big[\frac12 \sup_{\g \in [\ul\si,\ol\si]}\g^2 \b \frac{\ch_{\e}\we h}{h}  -ru_0\Big] +
C\d (h) +r\rho(\e) + \Big(C\big(r+\d (h)\big) +r\rho(\e) \Big) \cC^{\cP^{\rm UVM}}\big[\ch_{\e}\le h \big]\\
&\le & \frac12 \sup_{\g \in [\ul\si,\ol\si]}\g^2 \b  -ru_0 +
C\d (h) +r\rho(\e) + \Big(C\big(r+\d (h)\big) +r\rho(\e) + \frac12 \ol\si^2 |\b| \Big) \cC^{\cP^{\rm UVM}}\big[\ch_{\e}\le h \big].
\eeaa
By letting $h\rightarrow 0$, we get $ru_0 -r\rho(\e)-\sup_{\g \in [\ul\si,\ol\si]}\frac12 \g^2 \b ~\leq ~0$.
Finally, by letting $\e\rightarrow 0$, we obtain
\beaa
ru_0 -\sup_{\g \in [\ul\si,\ol\si]}\frac12 \g^2 \b ~\leq ~0.
\eeaa
\qed

\section{Appendix}\label{sec: some proofs}

\no {\bf Proof of Proposition \ref{prop: c1 stopping}}\quad The first result is easy, and we omit its proof. We decompose the proof in two steps.

\ms
\no \underline{\it Step 1.}\q We first prove that $\ol\cE^L[\ch_D]<\infty$. Without loss of generality, we may assume that $D=O_r$. Denote by $B^1$ the first entry of $B$.  Since 
\beaa
\ch_{O_r}\leq \ch^1_r:=\inf\{t\geq 0:|B^1_t|\geq r\},
\eeaa
it is enough to show that $\ol\cE^L[\ch^1_r]<\infty$. Thus,
without loss of generality, we may assume that the dimension $d=1$.

We first consider the following Dirichlet problem of ODE:
\be\label{ODE 1dim}
-L|\pa_x u|-\frac{1}{L}\pa^2_{xx}u -1=0,\q u(r)=u(-r)=0.
\ee
It is easy to verify that Equation \eqref{ODE 1dim} has a classical solution:
\beaa
u(x) = \frac{1}{L^3}\Big(e^{L^2r}-e^{L^2x}\Big)-\frac{1}{L}(R-x)
~\mbox{for}~0\le x\le r,\q\mbox{and}~u(x)=u(-x)~\mbox{for}~-r\le x\le 0.
\eeaa
Further, it is clear that $u$ is concave, so $u$ is also a classical solution to the equation:
\be\label{nonlinearODE 1dim}
-L|\pa_x u|-\frac12\sup_{\frac{2}{L}\le\b\le 2L}\b\pa^2_{xx}u -1=0,\q u(r)=u(-r)=0.
\ee
Then by It\^o's formula, we obtain
\beaa
0=u(B_{\ch_{O_r}}) = u_0 + \int_0^{\ch_{O_r}}\pa_x u(B_t)dB_t+\frac12\int_0^{\ch_{O_r}}\pa^2_{xx}u(B_t)d\langle B\rangle_t.
\eeaa
Recalling the definition of $\dbQ^{\a,\b}$ in \eqref{defn:PL} and taking the expectation on both sides, we have
\be\label{ito on u}
0=
u_0 + \dbE^{\dbQ^{\a,\b}}\Big[\int_0^{\ch_{O_r}}\big(\a_t\pa_x u(B_t)+\frac12 \b^2_t\pa_{xx}^2u(B_t)\big)dt\Big] 
\q\mbox{for all}\q 
\|\a\|\le L,\frac{2}{L}\leq \b_\cd\leq 2L
\ee
Since $u$ is a solution of Equation \eqref{nonlinearODE 1dim}, we have
\beaa
\dbE^{\dbQ^{\a,\b}}\Big[\int_0^{\ch_{O_r}}\big(\a_t\pa_x u(B_t)+\frac12 \b^2_t\pa_{xx}^2u(B_t)\big)dt\Big] 
\leq
-\dbE^{\dbQ^{\a,\b}}[\ch_{O_r}]
\eeaa
Hence $u_0\geq \ol\cE^L[\ch_{O_r}]$. On the other hand, taking $\a^*:=L{\rm sgn}\big(\pa_x u(B_t)\big)$ and $\b^*:=\sqrt{\frac{2}{L}}$, we obtain from \eqref{nonlinearODE 1dim} and \eqref{ito on u} that
$$u_0 = \dbE^{\dbQ^{\a^*,\b^*}}[\ch_{O_r}].$$
So, we have proved that $u_0 = \ol\cE^L[\ch_{O_r}]$. Consequently, $\ol\cE^L[\ch_{O_r}]<\infty$.

\ms
\no \underline{\it Step 2.}\q Note that
\[
\mathcal{C}^{L}\left[\ch_{D}\geq T\right]\leq\frac{\mathcal{\overline{E}}^{L}\left[\ch_{D}\right]}{T}.
\]
By the result of Step 1, we have $\mathcal{C}^{L}\left[\ch_{D}\geq T\right]\leq\frac{C}{T}$, and then $\lim_{T\rightarrow\infty}\mathcal{C}^{L}\left[\ch_{D}\geq T\right]=0$.
Further,
\bea\label{app-seq}
\mathcal{C}^{L}\left[\ch_{n}<\ch_{D}\right] & \leq & \mathcal{C}^{L}\left[\ch_{n}<\ch_{D};\ch_{D}\leq T\right]+\mathcal{C}^{L}\left[\ch_{n}<\ch_{D};\ch_{D}>T\right] \notag \\
 & \leq & \mathcal{C}^{L}\left[\ch_{n}<T\right]+\mathcal{C}^{L}\left[\ch_{D}>T\right].
\eea
We conclude that $\lim_{n\rightarrow\infty}\mathcal{C}^{L}\left[\ch_{n}<\ch_{D}\right]=0$.

\no Further, define $\hat{D}:=\cup_{x\in D}D^{x}.$ Note that
$\ch_{D}^{x}~\leq ~ \ch_{\hat{D}}$ for all $x\in D$.
Hence we have
\beaa
\sup_{x\in D}\mathcal{C}^{L}\left[\ch_{D}^x \geq T\right]
&\leq &
\mathcal{C}^{L}\left[\ch_{\hat{D}}\geq T\right]\rightarrow 0.
\eeaa
Together with \eqref{app-seq}, we obtain
$\lim_{n\rightarrow\infty}\sup_{x\in D}\mathcal{C}^{L}\left[\ch_{n}<\ch^x_{D}\right]=0$.
\qed

\vspace{3mm}

\no {\bf Proof of Lemma \ref{lem: comp on ball}}\quad
For simplicity, denote
\beaa
&\q g^i:=G(\o^i,\cd,\cd,\cd)~ (i=1,2), \q c_0:=\rho\big(d^e(\o^1,\o^2)\big)~ (\ge |g^1-g^2|),&\\
& {\bf L}^i u := -g^i (u, \pa_x u, \pa^2_{xx} u) ~(i=1,2), \q\q\mbox{and}\q \d h:=h^1-h^2 . &
\eeaa
By standard argument, one can easily verify that function 
$$w(x):=\overline{\mathcal{E}}^{L_{0}}\Big[\delta h^{+}(x+B_{\ch^x_{D}})+c_{0}\ch_{D}^{x} \Big]$$
is a C-L viscosity solution of the nonlinear PDE:
\[
-c_{0}-L_{0}|\pa_x w|-\frac{1}{2}\sup_{\sqrt{\frac{2}{L_{0}}}I_{d}\leq\g\leq\sqrt{2L_{0}}I_{d}}\g^{2}:\pa_{xx}^{2}w=0
~\mbox{on}~D,
\q\mbox{and}~w=(\delta h)^{+}\text{ on }\partial D.
\]
Let $K$ be a smooth nonnegative kernel with unit total mass. For
all $\eta>0$, we define the mollification $w^{\eta}:=w*K^{\eta}$
of $w$. Then $w^{\eta}$ is smooth, and it follows from a convexity argument as in \cite{Krylov} that $w^{\eta}$ is a classic supersolution of
\be\label{eq: w eta}
-c_{0}-L_{0}|\pa_x w^{\eta}|-\frac{1}{2}\sup_{\sqrt{\frac{2}{L_{0}}}I_{d}\leq \g \leq\sqrt{2L_{0}}I_{d}}\g^{2}:\pa_{xx}^{2} w^{\eta} \geq 0 \text{ on }D,
\q\text{and}\ w^{\eta}=(\delta h)^{+}*K^{\eta}\text{ on }\partial D.
\ee
We claim that
\[
\bar{w}^{\eta}+v^{2}\ \text{is a C-L viscosity supersolution to the PDE with generator}\ g^{1},
\]
where $\bar{w}^{\eta}:=w^{\eta}+\d$, with $\d:=\max_{x\in \pa D}|w^{\eta}(x)-(\delta h)^{+}(x)|$.
Then we note that 
\beaa
\bar{w}^{\eta}+v^{2}\geq w^{\eta}+h^{2}+\d \geq h^{1}=v^{1}\q\mbox{on}~~\partial D.
\eeaa
 By comparison principle for the C-L viscosity solutions of PDE's, we have $\bar{w}^{\eta}+v^{2}\geq v^{1}$
on $\mathrm{cl}(D)$. Setting $\eta\rightarrow0$, we obtain that
$v^1-v^2\leq w$. The desired result follows.

It remains to prove that $\bar{w}^\eta+v^{2}$ is a C-L viscosity supersolution of the
PDE with generator $g^{1}$. Let $x_{0}\in D$, $\phi\in C^{2}(D)$
be such that $0=(\phi-\bar{w}^{\eta}-v^{2})(x_{0})=\max\left(\phi-\bar{w}^{\eta}-v^{2}\right)$.
Then, it follows from the viscosity supersolution property of $v^{2}$
that ${\bf L}^{2}(\phi-\bar{w}^{\eta})(x_{0})\geq0$. Hence,
at the point $x_{0}$, by (\ref{eq: w eta}) we have
\begin{eqnarray*}
{\bf L}^{1}\phi & \geq & {\bf L}^{1}\phi-{\bf L}^{2}(\phi-\bar{w}^{\eta})\\
 & = & -g^{1}(\phi,\pa_x \phi,\pa_{xx}^{2}\phi)+g^{2}(\phi-\bar{w}^{\eta},\pa_x (\phi-\bar{w}^{\eta}),\pa_{xx}^{2}(\phi-\bar{w}^{\eta}))\\
 & \geq & -g^{1}(\phi,\pa_x \phi, \pa_{xx}^{2}\phi)+g^{2}(\phi,\pa_x (\phi-\bar{w}^{\eta}),\pa_{xx}^{2}(\phi-\bar{w}^{\eta}))\\
 & \geq & -c_{0} - L_{0}|\pa_x w^{\eta}| - \frac{1}{2}\sup_{\sqrt{\frac{2}{L_{0}}}I_{d}\leq \g \leq\sqrt{2L_{0}}I_{d}}\g^{2}:\pa_{xx}^{2}w^{\eta}\\
 & \geq & 0,
\end{eqnarray*}
where the last inequality is due to \eqref{eq: w eta}.
\qed

\begin{prop}\label{prop: estimate hitting}
For all $n\geq 1$, there exists a modulus of continuity $\rho$ such that
\beaa
\mathcal{\overline{E}}^{L}\left[|\ch_{Q}^{x_1}-\ch_{Q}^{x_{2}}|\right]\leq\rho\left(|x_1-x_2|\right).
\eeaa
\end{prop}
\proof
By the tower property, we have
\beaa
\ol\cE^L\Big[|\ch^{x_1}_Q-\ch^{x_2}_Q|\Big]
& \le & 
\ol\cE^L\Big[|\ch^{x_1}_Q-\ch^{x_2}_Q| 1_{\{\ch^{x_1}_Q\le\ch^{x_2}_Q\}}\Big]+
\ol\cE^L\Big[|\ch^{x_1}_Q-\ch^{x_2}_Q| 1_{\{\ch^{x_1}_Q>\ch^{x_2}_Q\}}\Big]\\
&\le &
\ol\cE^L\Big[\ol\cE^L\Big[\ch_Q^{x_2  +B_{\ch_Q^{x_1}}}\Big] 1_{\{\ch^{x_1}_Q\le\ch^{x_2}_Q\}}\Big]
+
\ol\cE^L\Big[\ol\cE^L\Big[\ch_Q^{x_1+ B_{\ch_Q^{x_2}}}\Big] 1_{\{\ch^{x_1}_Q>\ch^{x_2}_Q\}}\Big].
\eeaa
So, it suffices to show that there exists a modulus of continuity  $\rho$ such that
\beaa
\ol\cE^L\Big[\ch_Q^{x_2 + \o'_{\ch_Q^{x_1}}}\Big]
~\leq~
\rho\Big(|x_1-x_2|\Big),
\q \mbox{for all}~\o'~\mbox{such that}~\ch^{x_1}_Q(\o')\leq \ch^{x_2}_Q(\o').
\eeaa
Denote $y_i := x_i + \o'_{\ch^{\o^1}_Q}$ for $i=1,2$. Note that
$$|y_1-y_2|=|x_1-x_2|,\q
y_1\in\pa Q,~y_2\in Q.
$$

In the case of the dimension $d=1$, we may assume that $Q=[0,h]$ for some $h>0$. Next, consider the Dirichlet problem of ODE:
\be\label{ODE again}
-L|\pa_x u|-\frac12\sup_{\frac{2}{L}\le\b\le 2L}\b\pa^2_{xx}u-1=0
\q\mbox{and}\q
u(-\frac{h}{2})=u(\frac{h}{2})=0
\ee
Then, as in the proof of Proposition \ref{prop: c1 stopping}} above, we can prove that Equation \eqref{ODE again} has a classical solution $u$ and 
$$\ol\cE^L \big[\ch_{Q}^{y_2}\big]
~=~
u\Big(\frac{h}{2}-|x_1-x_2|\Big)
~=~ u\Big(\frac{h}{2}-|x_1-x_2|\Big) - u\big(\frac{h}{2}\big)
~\leq~
\rho\big(|x_1-x_2|\big),$$
where $\rho$ is the modulus of continuity of $u$.

In the case $d>1$, we need the following discussion. Since $Q$ is bounded and convex, there exists a $d$-dimensional open cube $\widehat Q$ such that 
$Q\subset \widehat Q$, $d(y_2,\pa\widehat Q) \le |y_1-y_2| = |x_1-x_2|$ and there is a unique point 
$y^* \in \pa \widehat Q$ such that  $d(y_2,\pa\widehat Q) = |y_2-y^*|$.
Since $\ch^{y_2}_{Q}\le \ch^{y_2}_{\widehat Q}$, it is enough to prove
\be\label{enoughwidehat}
\ol\cE^L\Big[\ch_{\widehat Q}^{y_2}\Big]
~\leq~
\rho\Big(|x_1-x_2|\Big).
\ee
Denote the unit vector $e^* := \frac{y^*-y_2}{|y^*-y_2|} $. Note that
\be\label{lineincube}
y_2 + |y^*-y_2| e^* \in \pa \widehat Q
\q\mbox{and there is a constant $\ell>0$ such that}~
y_2 - \ell e^* \in \pa \widehat Q
\ee
Denote a new stopping time
\beaa
\ch^*:= \inf \{t\ge 0: B\cd e^* \notin (-\ell, |y^*-y_2|) \}.
\eeaa
Since $\widehat Q$ is a cube, it follows from \eqref{lineincube} that $\ch_{\widehat Q}^{y_2}\le \ch^*$. Since $B\cd e^*$ takes values in $\dbR^1$, it follows from the previous result in the case $d=1$ that
\beaa
\ol\cE^L\big[\ch^*\big] ~\le ~ \rho \big(|y^*-y_2|\big) 
~\le  ~\rho \big(|x_1-x_2|\big),\q\mbox{for some modulus of continuity}~\rho.
\eeaa 
 Together with the fact  $\ch_{\widehat Q}^{y_2}\le \ch^*$, we finally obtain \eqref{enoughwidehat}.
 \qed

\end{document}